\documentclass[11pt,twoside]{amsart}
\usepackage{mathtools,amsmath,amssymb}
\usepackage{mathrsfs}

\usepackage[a4paper,top=1in, bottom=1in, left=1in, right=1in]{geometry} % resize page
\usepackage{graphicx} % Required for inserting images
\usepackage[colorlinks=true,linkcolor=blue,citecolor=black]{hyperref} % Enable hyperlinks
\usepackage{cleveref}
\usepackage[backend=biber,backref=true,doi=false,url=false,maxnames=9,sorting=nyt]{biblatex} % Enable backref
\renewbibmacro*{pageref}{%
  \iflistundef{pageref}
    {}
    {\addspace% Custom separator
     \mkbibbrackets{\printlist[pageref][-\value{listtotal}]{pageref}}}} 
\usepackage{lipsum}% for rescaling tikzcd 
\usepackage{adjustbox} % for rescaling tikzcd
\usepackage{tikz-cd}
\usepackage{mathptmx}
\usepackage{float}
\usepackage{amssymb}
\usepackage{amsthm}
\usepackage{amsmath}
\usepackage{mathtools}
\usepackage{soul}
\usepackage{nicematrix}
\usepackage{array}   
\usepackage{booktabs}
\usepackage{float}
\usepackage{orcidlink}
\usepackage[toc,page]{appendix}
\usepackage{caption}
\usepackage{mathrsfs}     
\newcommand{\GL}{\mathrm{GL}}
\newcommand{\M}{\mathrm{M}}
\newcommand{\im}{\mathrm{Im}}

\newcommand{\upchi}{\raisebox{0.4ex}{$\chi$}}
\newcommand{\Char}{\upchi}
\newcommand{\Z}{\mathbb{Z}}

\newcommand{\R}{\mathscr{O}_2}

\newcommand{\Fq}{\mathbb{F}_q}
\newcommand{\Zen}{\mathscr{Z}}
\usepackage{thmtools, thm-restate}

\newtheorem{lemma}{Lemma}[section]
\newtheorem{proposition}[lemma]{Proposition}
\newtheorem{corollary}[lemma]{Corollary}
\theoremstyle{definition} % Sets non-italic, upright font
\newtheorem{definition}[lemma]{Definition}
\newtheorem{example}[lemma]{Example}
\newtheorem{remark}[lemma]{Remark}
\newtheorem{question}{Question}
\title[Power maps on $\mathrm{GL}_n(\mathscr O_2)$]{Power maps on  General Linear groups over finite principal ideal local rings of length two}
\date{\today}
\author[Panja]{Saikat Panja\orcidlink{0000-0002-9639-3122}}
\email[(Panja)]{panjasaikat300@gmail.com}
\address{Indian Statistical Institute, Bengaluru Centre, 8th Mile, Mysore Rd, RVCE Post, Gnana Bharathi, Bengaluru, Karnataka 560059, India}
\author[Roy]{Ayon Roy}
\email[(Roy)]{ayonroy1999@gmail.com}
\address{Indian Institute of Science Education and Research Pune, Dr. Homi Bhabha Road, Pashan, Pune 411 008, India}
\author[Singh]{Anupam Singh}
\email[(Singh)]{anupamk18@gmail.com}
\address{Indian Institute of Science Education and Research Pune, Dr. Homi Bhabha Road, Pashan, Pune 411 008, India}
\thanks{Panja is supported by an NBHM postdoctoral fellowship, file number ending at R\&D-II/6746. Roy is supported by an IISER Pune PhD Fellowship. Singh is supported by an NBHM research grant 02011/23/2023/NBHM(RP)/RDII/5955}
\date{\today}
\dedicatory{We dedicate this paper to Maneesh Thakur for his wonderful mathematics.}
\begin{document}
\begin{abstract}
Word maps have been studied for matrix groups over a field. We initiate the study of problems related to word maps in the context of the group $\mathrm{GL}_n(\mathscr O_2)$, where $\mathscr O_2$ is a finite local principal ideal ring of length two (e.g.  $\mathbb{Z}/p^2\mathbb{Z}$ and $\mathbb F_q[t]/\langle t^2\rangle$). We study the power map $g\mapsto g^L$, where $L$ is a positive integer. We consider $L$ to be coprime to $p$ (an odd prime), the characteristic of the residue field $k$ of $\mathscr O_2$. We classify all the elements in the image, whose mod-$\mathfrak m$ reduction in $\mathrm{GL}_n(k)$ are either regular semisimple or cyclic, where $\mathfrak m$ is the unique maximal ideal of $\mathscr O_2$. 
Our main tool is a Hensel lifting for polynomial equations over $\mathrm{M}_n(\mathscr O_2)$, which we establish in this work. 

A central contribution of this work is the construction of canonical forms for certain natural classes of matrices over $\mathscr O_2$. As applications, we derive explicit generating functions for the probabilities that a random element of $\mathrm{GL}_n(\mathscr O_2)$ is regular semisimple, $L$-power regular semisimple, compatible cyclic, or $L$-power compatible cyclic.
\end{abstract}
\maketitle
\tableofcontents

\section{Introduction}\label{sec:intro}
In 1951, {\O}ystein Ore proved that every element of the alternating group $A_n$ is a commutator and conjectured that the same holds for all finite non-abelian simple groups; see \cite{Ore1951}. In the early 1960s, R.~C.~Thompson verified the conjecture for the groups $\mathrm{PSL}_n(q)$, and further progress was made by Gow and O.~Bonten in related cases; see \cite{Thompson1961, Thompson1962}. In 1984, J.~Neubüser, H.~Pahlings, and E.~Cleuvers confirmed the conjecture for the sporadic simple groups. This was followed in 1993 by Bonten’s result for all exceptional groups of Lie rank at most $4$; see \cite{Bonten1993}. In 1998, E.~W.~Ellers and N.~L.~Gordeev proved the conjecture for all finite simple groups of Lie type over $\mathbb{F}_q$, assuming $q \geq 8$. The conjecture was finally settled in its entirety in 2010 by Liebeck, O'Brien, Shalev, and Tiep \cite{LiebeckBrienShalev10}, through the use of advanced tools from the character theory of finite groups of Lie type and asymptotic group theory.

This, and the Waring problem, opened up a whole new world of investigations, namely the question of finding images of \emph{word maps} on several groups, both finite and infinite. Given a \emph{word} $w$, an element of the free group $F_r$ on $r$ generators, and a group $G$, one defines a map $\widetilde{w} \colon G^r \rightarrow G$ by evaluation. Two fundamental questions being extensively studied in this subject are (a) what is the image of $\widetilde{w}$ on $G$, i.e., describe the set $\widetilde{w}(G^r) = \{g\in G\mid \text{there exist } h_1,\ldots, h_r \in G, \widetilde{w}(h_1, \ldots, h_r) = g\}$, and (b) does there exist $k_w$, a positive integer, such that the subgroup $\langle\widetilde{w}(G^r)\rangle = \widetilde{w}(G^r)^{k_w}$; where for a subset $S\subseteq G$, $S^\ell=\{s_1s_2\cdots s_\ell \mid s_i\in S\}$. Hereafter, the set $\widetilde{w}(G^r)$ is denoted by $w(G)$, as is the tradition in the subject. In subsequent years, results on word maps — particularly for groups of Lie type (not necessarily simple) — have primarily focused on the case over fields. As a full account of related work is beyond the scope of this article, we include a few selected references and acknowledge that many important contributions remain unmentioned. Borel in 1983 (independently by Larsen in 2004; \cite{Larsen2004}) proved that, given a semisimple algebraic group $G$ and a word map $w\colon G^r \rightarrow G$, it is a dominant map (that is, the image is dense in $G$ for the Zariski topology), and hence one gets that $w(G)^2 = G$; see \cite{Borel1983}.

The results of Liebeck, O'Brien, Shalev, and Tiep \cite{LiebeckBrienShalev10} were extended to full generality in \cite{Shalev2009}, which states that if $w\neq 1$ is a non-trivial group word, there exists $N(w)$ such that for every non-abelian finite simple group $G$ with $|G|> N(w)$ we have $w(G)^3 = G$, extending the results of \cite{LiebeckShalev2001}. The number $3$ was further reduced to $2$ in \cite{LarsenShalevTiep2011Annals}, generalizing the findings about $A_n$, the alternating group, in \cite{LarsenShalev2009JAMS}. The result $w(G)^2 = G$ is the best possible, as for the word $w = x_1^2$, $w(A_5)\neq A_5$. We must mention that the Waring-like problems have been studied in the case of Lie groups and Chevalley groups in \cite{HuiLarsenShalev15}, unipotent algebraic groups in \cite{Larsen19unipotent}, residually finite groups in \cite{LarsenShalev18residual}, $p$-adic and Adelic groups in \cite{AvniShalev13}, discrete group $\mathrm{SL_n(\Z)}$ in \cite{AvniMeiri2019} etc.

In recent times, the image of power maps on classical groups has attracted attention within statistical group theory. The investigation began with the finite general linear groups in \cite{KunduSinghGL}, and was subsequently extended to orthogonal and symplectic groups in \cite{PanjaSinghSymplectic25}, and to unitary groups in \cite{PanjaSinghUnitary24}. Related asymptotic results on the distribution of powers among regular semisimple, semisimple, and regular elements in finite reductive groups are found in \cite{KulshresthaKunduSingh2022}. Questions concerning the fibers of such maps, which naturally follow from studying their images, are explored in \cite{panja2024a}. Power maps —particularly the squaring map — also play a role in enumerating real conjugacy classes (those containing elements conjugate to their inverses), see \cite{panja2024b}. Moreover, they provide examples of word maps with dense \emph{image ratios}; for further details, see \cite{panja2024c}. For a survey regarding the study of power maps on other groups, one can see \cite{PanjaSingh2024Survey}. We also note that power maps, in the context of Lie groups and algebraic groups, are closely tied to the property of the group being exponential; this relationship has been studied by Chatterjee and Steinberg (see \cite{Chatterjee2002,Chatterjee2003,Steinberg2003}).

As discussed above, the word problems, in particular the power map, have been studied extensively on matrix groups over fields. We aim to extend this line of inquiry by addressing a basic yet fundamental question in a broader setting:
\begin{question}
Let $\R$ be a finite local principal ideal ring of length two, and its residue field $k$ is of characteristic $p$. For $L>0$, an integer, consider the power map $\Phi_L \colon \GL_n(\R) \rightarrow \GL_n(\R)$ given by $x \mapsto x^L$. What is the image of $\Phi_L$? One can further ask this question for other matrix groups. 
\end{question}
We address the above question for the classes of regular semisimple and cyclic elements of $\GL_n(\R)$, under the assumption that $\gcd(L, p) = 1$; see the relevant sections for the definitions of regular semisimple and cyclic elements. Our approach is to exploit the known result over the field by going modulo the maximal ideal to the residue field, and use Hensel lifting. Note that the complication here is at several levels, as the canonical form for the matrix may not lift in a compatible fashion, and the uniqueness of the polynomial invariants of the similarity classes might not behave well, either. Of all the results presented in this article, the following four constitute the main contributions.

\begin{restatable}{theorem}{theoremone}\label{thm:reg-sem-poly}
Let $\R$ be a finite local principal ideal ring of length two, and the corresponding residue field $k$ has characteristic $p$, an odd prime. Let $L>0$ be an integer coprime to $p$. Then, a regular semisimple element $A\in \GL_n(\R)$ is an $L$-th power if and only if each fundamental irreducible factor of the characteristic polynomial $\Char_{\R, A}(t)$ is an $L$-power polynomial.
\end{restatable}

\begin{restatable}{theorem}{theoremtwo}\label{thm:cyc}
Let $\R$ be a finite local principal ideal ring of length two, and the corresponding residue field $k$ has characteristic $p$, an odd prime.  
Let $L>0$ be an integer coprime to $p$. Then, a compatible cyclic element $A\in \GL_n(\R)$ is an $L$-th power if and only if each fundamental irreducible factor of $\Char_{\R, A}(t)$ is an $L$-power polynomial. 
\end{restatable}
\noindent In both of the theorems above, the statements are analogous to the theorems for a field. Recall that in the field case, the criteria for an invertible matrix to be $L$-th power are given in terms of the corresponding polynomial invariants being $L$-th power.

\begin{restatable}{theorem}{theoremthree}\label{thm:gen-fun}
Let $\R$ be a finite local principal ideal ring of length two, with unique maximal ideal $\mathfrak{m}$. Let $p$ be the characteristic of its residue field $k$, which is an odd prime. Fix an integer $L > 0$ such that $\gcd(L, p) = 1$. Let $s_n$ denote the probability that a randomly chosen element of $\GL_n(\R)$ is regular semisimple, and let $s_{n, L}$ denote the probability that a randomly chosen element of $\GL_n(\R) \cap \im(\Phi_L)$ is regular semisimple. Then the generating functions for these probabilities admit the following factorization:
\begin{align}
1+\sum\limits_{n=1}^{\infty} s_nz^n& = \prod_{d=1}^{\infty}\left(1+\dfrac{z^d}{|\mathfrak{m}|^d(q^d-1)}\right)^{|\mathfrak m|^d N(q,d)},\\
1+\sum\limits_{n=1}^{\infty} s_{n,L}z^n& = \prod_{d=1}^{\infty}\left(1+\dfrac{z^d}{|\mathfrak{m}|^d(q^d-1)}\right)^{N_{\R,L}(q,d)},
\end{align}
where $q$ is the order of the residue field.
\end{restatable}
\begin{restatable}{theorem}{theoremfour}\label{thm:gen-fun-cyc}
Let $\R$ be a finite local principal ideal ring of length two, with unique maximal ideal $\mathfrak{m}$. Let $p$ be the characteristic of its residue field $k$, which is an odd prime. Fix an integer $L > 0$ such that $\gcd(L, p) = 1$. Let $r_n$ denote the probability that a randomly chosen element of $\GL_n(\R)$ is compatible cyclic, and let $r_{n, L}$ denote the probability that a randomly chosen element of $\GL_n(\R) \cap \im(\Phi_L)$ is compatible cyclic. Then the generating functions for these probabilities admit the following factorization:
\begin{align}
1+\sum\limits_{n=1}^{\infty} r_nz^n& = \prod_{d=1}^{\infty}\left(1+\sum\limits_{s=1}^{\infty}\dfrac{z^{ds}}{|\mathfrak{m}|^{ds}q^{(s-1)d}(q^d-1)}\right)^{|\mathfrak m|^d N(q,d)},\\
1+\sum\limits_{n=1}^{\infty} r_{n,L}z^n& = \prod_{d=1}^{\infty}\left(1+\sum\limits_{s=1}^{\infty}\dfrac{z^{ds}}{|\mathfrak{m}|^{ds}q^{(s-1)d}(q^d-1)}\right)^{N_{\R,L}(q,d)},
\end{align}
where $q$ is the order of the residue field.
\end{restatable}
\noindent We set some notation below, which will be used throughout. 

%%%%%%%%%%%%%%%%%%%%%%
\subsection*{Notation and convention}
Throughout the article, $\R$ denotes a local principal ideal ring of length $2$, with its unique maximal ideal $\mathfrak{m}$. Some typical example of such a rings are $\mathbb{Z}/p^2\mathbb{Z}$ and $\mathbb F_p[t]/<t^2>$. Notation $\R^\times$ denotes the set of units of $\R$. Further, $p$ denotes the characteristic of the quotient field $k = R/\mathfrak m$, which will be taken to be an odd prime. We will be mostly dealing with $L\geq 2$, an integer coprime to $p$. 

Let $\M_n(\R)$ be the set of all $n\times n$ matrices with entries from $\R$. The general linear group $\GL_n(\R)$ is the set of all elements $X\in \M_n(\R)$ such that $\det(X)\in \R^{\times}$. The quotient map $\theta\colon \R \rightarrow k =\R/ \mathfrak{m}$ induces a canonical map $\theta\colon \M_n(\R)\rightarrow \M_n(k)$ denoted by $X\mapsto \overline{X}$. Thus, we get a map $\theta\colon \GL_n(\R) \rightarrow \GL_n(k)$. Further, note that the quotient map $\theta\colon \R \rightarrow k$ also induces a surjective map of polynomials $\theta\colon \R[t]\rightarrow k[t]$. By abuse of notation, all the above maps are denoted simply by $\theta$, usually clear from the context. We say that an element $X\in \GL_n(\R)$ \emph{belongs to a conjugacy class of type $\mathbf{C}$} if the conjugacy class of $\overline{X} = \theta(X)$ in $\GL_n(k)$ is of type $\mathbf{C}$. Thus, an element $X\in \GL_n(\R)$ is said to be \emph{regular semisimple, semisimple or cyclic} if $\overline{X}$ is regular semisimple, semisimple or cyclic in $\GL_n(k)$, respectively. 
For a ring $\mathscr R$ and two elements $A, B\in\M_n(\mathscr R)$, the $\mathscr R$-conjugacy denoted as $A\sim_{\mathscr R} B$ means there exists $C\in\GL_n(\mathscr R)$ such that $A = CBC^{-1}$.
For $X\in\GL_n(\R)$ the natural $\R[t]$-module structure on $M = \R^n$ is denoted by $M^X$ and the corresponding $k[t]$ module structure on $k^n$ is denoted by ${M}^{\overline{X}}$. 

The power map $\Phi_L\colon \GL_n(\R) \rightarrow \GL_n(\R)$ given by  $X\mapsto X^L$ induces a power map $\overline{\Phi}_L\colon \GL_n(k) \rightarrow \GL_n(k)$. The $\R$-conjugacy class (resp. $k$-conjugacy class) of $X$ (resp. $\overline{X}$) is denoted  by $[X]_R$ (resp. $[\overline{X}]_k$). For a ring $\mathscr R$ and an element $A\in \GL_n(\mathscr R)$, the centralizer group is denoted by $\Zen_{\GL_n(\mathscr R)}(A)$. Uppercase letters such as $F(t), G(t), H(t)$ are used to denote a polynomial in $\R[t]$, and its image under the canonical map $\theta\colon \R[t]\rightarrow k[t]$ is denoted by lowercase letters such as $f(t), g(t), h(t)$ respectively.
So, if $f(t)\in k[t]$ is the image of $F[t]$ under the canonical map, then $\theta(F(t))= \overline{F}(t) =f(t)$.

For a commutative ring $\mathscr R$ with unity and a polynomial $\mathfrak f(t)=t^n+\sum\limits_{i=0}^{n-1} c_it^i\in \mathscr R[t]$, the companion matrix $C_{\mathfrak f}\in\M_n(\mathscr R)$ of degree $n$ is defined to be the matrix $$\begin{pmatrix}
    0 & 0 & \dots & 0 & -c_0 \\
1 & 0 & \dots & 0 & -c_1 \\
0 & 1 & \dots & 0 & -c_2 \\
\vdots & \vdots & \ddots & \vdots & \vdots \\
0 & 0 & \dots & 1 & -c_{n-1}
\end{pmatrix}.$$

\subsection*{Organization of the article} 
Having provided a brief overview of the main question and existing results, we now present the preparatory material in \Cref{sec:result-polynomials}, which includes results about the polynomials in $\mathscr O_\ell[t]$.
This section also contains (previously known) several versions of Hensel's lemma, which we require for our work. \Cref{sec:prep} describes characteristic polynomial, minimal polynomial of matrices in $\GL_n(\R)$, and further describe (previously known) conjugacy classes of $\GL_2(\Z/p^{\ell}\Z)$. In \Cref{sec:separable}, we study regular semisimple matrices in $\GL_n(\R)$. One of the main results, \Cref{thm:reg-sem-poly}, appears here; it characterizes when a matrix $A \in \GL_n(\R)$ lies in the image of $\Phi_L$, in terms of its characteristic polynomial.
\Cref{sec:cyc} is devoted to the study of special classes of cyclic matrices in $\GL_n(\R)$, namely compatible cyclic. 
We establish a canonical form for these matrices, presented in \Cref{prop-canonical-form-type}.
Using this form, we establish a criterion for a cyclic matrix to lie in $\im(\Phi_L)$ in terms of properties of its characteristic (and minimal) polynomial, see \Cref{thm:cyc}.
Next, in \Cref{sec:counting}, we prove \Cref{thm:gen-fun} and \Cref{thm:gen-fun-cyc}, after first deriving a formula for the number of monic fundamental irreducible polynomials in $\R[t]$ of a given degree.  
In \Cref{sec:res-cond-L-p} we provide examples illustrating the necessity of the hypothesis $\gcd(L,p)=1$. Finally, we conclude the article with \Cref{res-conc-rem}, where we propose some questions for further investigation.
\color{black}
\subsection*{Acknowledgment} 
%This work was initiated at the Harish-Chandra Research Institute (HRI), where Panja was a postdoctoral fellow. A significant part of the writing was completed during his visit to IISER Pune in August 2025. He gratefully acknowledges HRI, IISER Pune, and his current institute, ISI Bangalore, for providing a supportive research environment. 
We thank Hassain M.\ for helpful discussions on conjugacy classes in $\GL_2(\Z/p^\ell\Z)$. We also thank B. Sury for his interest in this work.

\section{Properties of polynomials in $\mathscr O_\ell[t]$}\label{sec:result-polynomials}

This section reviews background materials on polynomials over a local principal ideal ring of length $\ell$, denoted as $\mathscr O_\ell$. It begins with completely primary rings, discusses different versions of Hensel's lemma, and collects information about factorization over $\mathscr O_\ell$, which we require later for our work.

\subsection{Completely primary ring}
A \emph{completely primary ring} is a ring whose nilradical is a maximal ideal. This enables us to use the theory of complete local rings, as developed by Snapper in \cite{Snapper1950}, \cite{Snapper1951a}, \cite{Snapper1951b}, and \cite{Snapper1952}. A local principal ideal ring of length $\ell$ is an example of a completely primary ring. For our work in this article, we require certain results about a specific factorization of a monic polynomial in $\mathscr O_\ell[t]$, which we mention here.  

A polynomial $F(t) \in\mathscr O_\ell[t]$ is called a \emph{fundamental irreducible} of $\mathscr O_\ell[t]$ if $\theta(F(t)) = \overline{ F}(t)$ is an irreducible element of $k[t]$, see \cite[Definition 2.1]{Snapper1950}. 
Recall, if $\mathscr R$ is a commutative ring with identity, an element $\mathfrak{f}\in\mathscr R$ is called an \emph{irreducible} element if  $\mathfrak{f}=\mathfrak{g} \mathfrak{h}$ implies either $\mathfrak{g}$ or $\mathfrak{h}$ is a unit in $\mathscr R$.
An element $\mathfrak{f}\in\mathscr R$ is \emph{primary} element if the ideal $\langle \mathfrak{f}\rangle$ is a primary ideal (see \cite[p. 50]{AtiyahMacBook69}). 
Two ideals $I$ and $J$ of $\mathscr R$ are called \emph{co-prime or relatively divisorless} (as per Snapper) if their ideal sum $(I, J)=\mathscr R$. 
Two elements $\mathfrak{f},\mathfrak{g}$ are called \emph{associated elements} if their respective principal ideals $\langle\mathfrak{f}\rangle$ and $\langle\mathfrak{g}\rangle$ are equal in $\mathscr R$. 
With this in mind, we have the following results about factorization in primary rings (see \cite[Theorem 5.1]{Snapper1950}).

\begin{lemma}\label{lem:snap-factor}
Let $\mathscr R$ be a ring, $N(\mathscr R)$ be its nilradical, and $\mathscr R/N(\mathscr R)$ be an integral, principal ideal ring. Then the following are true.
\begin{enumerate}
\item Every non-nilpotent element $\alpha$ of $\mathscr R$ can be factored as $\alpha = \delta \sigma_1\sigma_2\ldots\sigma_n$, where $\delta$ is a unit and $\sigma_1,\ldots,\sigma_n$ are primary, not nilpotent non-units, which are coprime in pairs.
\item If $\delta \sigma_1 \ldots \sigma_n = \delta' \sigma_1' \ldots \sigma_{n'}'$, where $\delta$ and $\delta'$ are units, $\sigma_1, \ldots, \sigma_n$ are primary non-units which are coprime in pairs and the same is true for $\sigma'_1,\ldots, \sigma'_{n'}$; then $n = n'$ and after a suitable reordering, $\sigma_i$ is associated with $\sigma_i'$ for $i=1, 2, \ldots, n$. If $n> 1$, the elements $\sigma_1, \ldots,\sigma_n, \sigma'_1, \ldots, \sigma'_{n'}$ are necessarily not nilpotent.
\end{enumerate}
\end{lemma}
\noindent Further, we have \cite[p. 673]{Snapper1950},
\begin{lemma}\label{lem:primary-modulo}
An element $\alpha\in \mathscr R$ is primary and non-nilpotent if and only if $\overline{\alpha}$ is a primary nonzero element of $\mathscr R/N(\mathscr R)$.
\end{lemma}

%%%%%%%%%%%%%%%5
\subsection{Several versions of Hensel's lemma}\label{subsec-hens-lem} 

Hensel's lemma appears in several well-known forms. Here, we state the versions required for our work. Without delving into details, we provide these versions here for the sake of completeness and give the appropriate reference for further details.

We begin with a version for polynomial factorisation over an appropriate ring, see \cite[Theorem 7.18]{Eisenbud95}. 
\begin{lemma}[Hensel's lemma version 1]\label{hensel-1}
Let $\mathscr R$ be a Noetherian ring, complete with respect to an ideal $\mathfrak{m}$. Let $F(t)$ be a polynomial in $\mathscr R[t]$ and $f(t)$ be the polynomial over $\mathscr R/\mathfrak{m}$ obtained by reducing $F(t)$ modulo $\mathfrak{m}$. Suppose $f(t)$ has a factorization $f(t) = g_1(t)g_2(t)$ in $\mathscr R/\mathfrak{m}[t]$ in such a way that $g_1(t)$ and $g_2(t)$ generate the unit ideal, and $g_1(t)$ is monic. Then, there is a unique factorization $$F(t)=G_1(t)G_2(t)\in \mathscr R[t]$$ such that $G_1(t)$ is monic and $G_i(t)$ reduces to $g_i(t)$ mod $\mathfrak{m}$ for $i=1, 2$.
\end{lemma}

\noindent Now, we collect the version from number theory (see \cite[Theorem 2.1]{ConradHensel}) as follows:
\begin{lemma}[Hensel's lemma version 2]\label{hensel-2}
Let $\Z_p$ denote the ring of $p$-adic integers. If $f(t) \in \Z_p[t]$ and $a \in \Z_p$ satisfies
 $$f(a)\equiv 0 \pmod{p},\text{ and } f'(a)\not\equiv 0 \pmod{p}$$ 
then, there is a unique $\alpha\in \Z_p$ such that $f(\alpha )=0$ in $\Z_p$ and $\alpha\equiv a\pmod{p}$.
\end{lemma}

\noindent Finally we have (see \cite[Lemma 2.3.8]{Singla2010}),  
\begin{lemma}[Hensel's lemma version 3]\label{hensel-3}
Let $h(t)$ be an irreducible polynomial over $\mathbb{F}_q[t]$. Then, for each positive integer $n$, there exists $q_n(t) \in \mathbb{F}_q[t]$ such that $q_n(t) \equiv t \pmod{h(t)}$, and $h(q_n(t)) \equiv 0\pmod{h(t)^n}$.
\end{lemma}

%%%%%%%%%%%%%%%%%%%%%%%
\subsection{Factorisation in $\mathscr O_\ell[t]$}
We are going to prove a unique factorization type result for some special classes of polynomials in $\mathscr O_\ell[t]$.
We begin with,
\begin{lemma}\label{lem:f-irred-F-irred}
Let $\mathscr O_\ell$ be a local principal ideal ring of length $\ell$, $\mathfrak{m}= \langle a\rangle$ be its unique maximal ideal, and $k$ be the residue field. Let $F(t)\in \mathscr O_\ell[t]$ be a fundamental irreducible polynomial. Then $F(t)$ must be irreducible in $\mathscr O_\ell[t]$.
\end{lemma}
\begin{proof}
Let $\overline{F}(t)=f(t) \in k[t]$. If possible, let there exist non-constant polynomials $G(t)$ and $H(t)$ in $\mathscr O_\ell[t]$ such that $F(t) = G(t)H(t)$. Since $f(t) = \overline{G}(t) \overline{H}(t)$ in $k[t]$ and $f(t)$ is irreducible, without loss of generality, we may assume that $\overline{H}(t) =c \in k$ is a unit. Then, $H(t) = m(t) + u$ for some $m(t)\in \mathfrak{m}[t]$ and $u\in \mathscr O_\ell$ a unit, satisfying $\theta(u) = c$. By \cite[Theorem 2.2]{ConradNilpotents}, $H(t)$ is a unit in $\mathscr O_\ell[t]$, whence $F(t)$ is irreducible.
\end{proof}
\noindent Next, we state a result describing the relationship between two coprime fundamental irreducible polynomials and their reductions under the map $\theta$.
\begin{lemma}\label{lem:F-G-fund-irred}
Let $\mathscr O_\ell$ be a local principal ideal ring of length $\ell$, with residue field $k$. Let $F(t), G(t)\in\mathscr O_\ell[t]$ be two polynomials. Then $F(t)$ and $G(t)$ are coprime if and only if their reductions $\overline{F}(t), \overline{G}(t)$ are coprime in $k[t]$.
\end{lemma}
\begin{proof}
Let $F(t)$ and $G(t)$ be coprime in $\mathscr O_\ell[t]$. By definition, there exist $A(t), B(t)\in \mathscr O_\ell$, such that $A(t)F(t) + B(t)G(t)=1$. On applying $\theta$, it reduces to $\overline{A}(t) \overline{F}(t) + \overline{B}(t)\overline{G}(t) = 1$; whence the forward direction.

For proving the other direction let $a(t), b(t)\in k[t]$ be such that $a(t)\overline{F}(t) + b(t)\overline{G}(t) =1$. Let $A(t), B(t)\in \mathscr O_\ell[t]$ be lifts of $a(t)$, $b(t)$ respectively. Since $\theta({A(t)F(t) + B(t)G(t)}) = 1$, the polynomial $A(t)F(t) + B(t)G(t)$ is a unit in $\mathscr O_\ell[t]$. This proves the existence of $\widetilde{A}(t), \widetilde{B}(t)$ such that $\widetilde{A}(t) F(t) + \widetilde{B}(t) G(t) = 1$.
\end{proof}

\begin{lemma}[Unique factorization using  monic fundamental irreducible polynomials]\label{lem:tech-1-thm-1}
Let $\mathscr O_{\ell}$ be a local principal ideal ring of length two and $F(t) \in\mathscr O_\ell[t]$ be a monic polynomial such that $\overline{F}(t) = f(t)$ is a separable polynomial in $k[t]$. Then $F(t)$ has a unique factorization into coprime fundamental irreducibles in $\mathscr O_\ell[t]$, up to a permutation of fundamental irreducible factors.
\end{lemma}
\begin{proof}
We will first show the existence of such a factorization and use \Cref{lem:snap-factor} to prove its uniqueness.

Let $F(t)\in \mathscr O_\ell[t]$ be a monic polynomial such that $\overline{F}(t) = f_1(t)f_2(t)\cdots f_r(t)$, where $f_i(t)$ are distinct irreducibles in $k[t]$ for $1\leq i\leq r$.
Without loss of generality, let us take $f_i(t)$ to be a monic polynomial (if not monic, divide by the leading coefficient) for some fixed $i$ and $g_1(t) = f_2(t)\cdots f_r(t)$.
Clearly $\gcd(f_1(t), g_1(t)) = 1$. 
Then by Hensel lemma version 1, \Cref{hensel-1}, we get polynomials $F_1(t)$ and $G_1(t)$ in $\mathscr O_\ell[t]$ such that $F(t) = F_1(t)G_1(t)$ where $\overline{F}_1(t) =f_1(t)$ and $\overline{G}_1(t) =g_1(t)$ in $k[t]$ and $F_1(t)$ is monic.
Now as $F_1(t)$ is monic and $f_1(t)$ is irreducible in $k[t]$, by \Cref{lem:f-irred-F-irred}, $F_1(t)$ must be irreducible in $\mathscr O_\ell[t]$ and hence is a fundamental irreducible.

Next, consider the polynomial ${g}_2(t) = f_3(t)\cdots f_r(t).$ 
Then, $\gcd(f_2(t), {g}_2(t))=1$. Moreover, $f_2(t)$ is monic irreducible and $g_1(t) = f_2(t){g}_2(t)$. Then by Hensel's lemma version 1, \Cref{hensel-1}, we get a factorization of $G_1(t)$ in $R[t]$ as $G_1(t) = F_2(t)G_2(t)$; where $F_2(t)$ is monic and $\overline{F}_2 =f_2, \overline{G}_2={g_2}$. 
Now, $F_2(t)$ must be fundamental irreducible. Hence, we get $F(t)=F_1(t) F_2(t) G_2(t)$. Continuing this process we achieve a factorization (with ordering) $F(t) =F_1(t)\cdots F_{r-1}(t)F_r(t)$ where each $F_i(t)$ is fundamentally irreducible and $F_1(t),\ldots, F_{r-1}(t)$ are monic. This implies $\deg(F_i(t)) = \deg(f_i(t))$ for $1\leq i\leq r$. As $F(t)$ is monic, this forces $F_r(t)$ to be monic (compare the leading term coefficient, as in \Cref{lem:f-irred-F-irred}). Note that all $F_i(t)$ are pairwise distinct, because otherwise $f_i(t)$ will not be distinct for $1\leq i\leq r$.

Now to prove uniqueness, if possible let there be an another decomposition of $F(t)$ in monic fundamental irreducibles as follows $$F(t)= H_1(t)H_2(t)\cdots H_{r'}(t)$$ such that all $H_i(t)$ are non-constant monic fundamental irreducible in $R[t]$ (this forces $r=r'$) and $\theta({H_i})=f_i$, $\deg(H_i(t))=\deg(f_i(t))$ for each $1\leq i\leq r$.
Consider $A=R[t]$ and $N(A)=\mathfrak{m}[t]$ in \Cref{lem:snap-factor}.
Clearly $F_1(t), F_2(t),\ldots, F_r(t)$ are primary (since $f_i$ is irreducible, hence $\langle f_i\rangle$ is prime, and thus \Cref{lem:primary-modulo} is applicable) non-units and they are relatively divisorless in pairs.
Moreover the same is true for $H_1(t),H_2(t),\ldots,H_r(t)$ also.
Then by \Cref{lem:snap-factor}  we get , after a suitable ordering $\langle F_i(t)\rangle=\langle H_i(t)\rangle$ for $1\leq i\leq r$. 
Therefore $H_i(t)=\alpha_i(t)F_i(t)$ where $\alpha_i(t)\in R[t]$ for $1\leq i\leq r$. 
As $H_i(t)$ is a monic irreducible polynomial, $\alpha_i(t)$ must be a unit in $\mathscr O_\ell[t]$.
This implies that $\alpha_i(t)=a_i+m_i(t)$ for each fixed $i$; where $a_i\in R^{\times}$ and $m_i[t]\in \mathfrak{m}[t].$ 

If possible, let $m_i(t)$ be non-constant.
As $F_i(t)$ is monic therefore the degree of $\alpha_i(t)F_i(t)$ exceeds the degree of $F_i(t)$. Which is a contradiction (since $\deg(H_i(t))$ is same as $\deg(F_i(t))$ by assumption). 
Hence, each $\alpha_i(t)$ should be a constant and that should be $1$.
Hence after the suitable ordering we get $F_i(t)=H_i(t)$ for $1\leq i\leq r$. So the decomposition is unique.
\end{proof}
The factors appearing in the factorization of $F(t) \in \mathscr O_\ell[t]$ for which $\theta(F)$ is separable, as in \Cref{lem:tech-1-thm-1}, are referred to as the \emph{fundamental factors} of $F$.
For later use, we record the following result.
\begin{lemma}\label{lem:cay-ham}
    Let $\mathscr R$ be a commutative ring with unity. For a matrix $A\in\M_n(\mathscr R)$ one has $\Char_{\mathscr R,A}(A)=\mathbf 0$.
\end{lemma}
We define an $L$-power polynomial as follows, which generalizes the definition in the case of a field, see \cite{KunduSinghGL}.
\begin{definition}
Let $\R$ be a local ring of finite length. A monic (fundamental) irreducible polynomial $F(t)\in \R[t]$ with $\deg F = d$ is called an $L$-power polynomial if $F(t^L)$ has a monic (fundamental) irreducible factor $G(t)\in R[t]$ of degree $d$.
\end{definition}
\begin{example}\label{ex:L-power-polynomial}
We note down some examples of $L$-power polynomials for different choices of local rings of length two. 
\begin{enumerate}
\item For $\R=\Z/9\Z$, the polynomial $F(t)=t^2 - 6t + 7 $ is a $2$-power polynomial, as $F(t^2)=(t^2 + 4t + 5)  (t^2 + 5t + 5)$.
\item Consider $\R=\mathbb{F}_3[u]/(u^2)$ and the polynomial $F(t)=t^2-2\Bar{u}t+(\Bar{u}+1)$, where $\Bar{u}=u+(u^2)$ in $\mathscr R$. 
Since $F(t^2)=t^4-2\Bar{u}t^2+(\Bar{u}+1)=F(t^2)=[t^2+(\Bar{u}+2)t+(\Bar{u}+2)][t^2-(\Bar{u}+2)t+(\Bar{u}+2)]$, and, $V(t)=t^2+(\Bar{u}+2)t+(\Bar{u}+2)$ is a degree $2$ monic irreducible factor of $F(t^2)$. 
As the reduction of $V(t)$ under $\theta$, the polynomial $v(t)=t^2+2t+2\in \mathbb{F}_3[t]$ is an irreducible polynomial, $F(t)$ is a 2 power polynomial in $\mathscr R[t]$.
\end{enumerate}
\end{example}
\section{Characteristic, minimal polynomials and conjugacy of elements}\label{sec:prep}
Characteristic and minimal polynomials play a crucial role in the study of conjugacy in $\GL_n(k)$, where $k$ is a field.  
However, these notions are far less understood for elements of $\GL_n(\R)$.  
In the first subsection, we establish some preliminary results on these polynomials in the context of $\GL_n(\R)$.  
In the latter part of this section, we provide an overview of conjugacy classes in $\GL_n(\R)$, with particular emphasis on the classification of conjugacy classes in $\GL_2(\mathbb{Z}/p^\ell\mathbb{Z})$ for $\ell \geq 2$.  

Since our interest lies in studying power maps, we approach them through their action on conjugacy classes.  
While the classification of conjugacy classes in $\GL_n(k)$ is carried out via canonical form theory, which relies on polynomial factorization over $k$, extending this approach to $\GL_n(\R)$ is considerably more difficult: both conjugacy classification and polynomial factorization over $\R$ present substantial challenges.
%%%%%%%%%%%%%%%%%%%%%%%%%%%%
\subsection{Characteristic and minimal polynomials for some special matrices}\label{subsec-min-char}

Let $\mathscr{R}$ be a ring. For a matrix $A\in \M_n(\mathscr R)$, consider the map $\Theta_A\colon \mathscr R[t]\rightarrow \M_n(\mathscr R)$ defined by the evaluation at $A$, viz. $\Theta_A(f(t))= f(A)$.
The kernel of this map is called the \emph{null ideal} of $A$ and is denoted by $N_A = \left\{F(t)\in \mathscr R[t] \mid F(A) = 0\right\}$. 
By \Cref{lem:cay-ham}, $N_A\neq\emptyset$.
In the literature, there are several cases studied when $N_A$ is principal, for example, the case of $\mathscr R$ being a commutative ring with unity has been considered in \cite{Brown2005}. Throughout this subsection, we will consider the case for $A\in\GL_n(\R)$ with the property that $\overline{A} \in \GL_n(k)$ has an irreducible characteristic polynomial. We will show that in this case, $N_A$ will be principal. 

Let us recall the notions of characteristic and minimal polynomials.
The \emph{characteristic polynomial} for $A\in \GL_n(\R)$ is defined by $\Char_{\R,A}(t)=\det(tI-A)\in \R[t]$. 
A polynomial $F(t)\in N_A$ such that $F(t)$ is monic and of the smallest degree will be called a  \emph{minimal polynomial} for $A$ over $\R$.
Note that it may not be unique in general (unlike the field case), but in the situation we deal with, it will be unique.
At the end of this section, we will see an example where the minimal polynomial is not uniquely determined.
To begin, we have the following lemma.
\begin{lemma}\label{lem:null-ideal-prin}
Let $A\in GL_n(\R)$. Suppose there exists monic $F(t)\in N_A$ such that $\deg(\overline{F}(t))=\deg(\mathrm{Min}_{k,\overline{A}}(t))$. 
Then, the null ideal $N_A$ is principal and is generated by the polynomial $F(t)\in \R[t]$.
\end{lemma}
\begin{proof}
Let $G(t)\in N_A$. Then, by \cite[Theorem 2.14]{JacobsonAlgebraBook85}, there exists polynomials $Z(t)$ and $R(t)$ with either $R(t)=0$ or $\deg(R(t)) < \deg(F(t))$ such that $$G(t) = Z(t)F(t) + R(t).$$ 
Now, if possible, let $R(t)$ be non-zero. Then $R(t) \in N_A$ and $\deg(R(t))<\deg(F(t))$. We consider two cases now. 
        
If possible, suppose at least one of the coefficients of $R(t)$ is a unit in $\R$. Then, after reduction, $\theta(R(t)) = {r(t)}$ will be an annihilating polynomial of $\overline{A}$ in $k[t]$. This contradicts the minimality of $\deg\overline{F}(t)$.
So, this case is not possible. 
        
Thus, we assume that $R(t) \in \mathfrak{m}[t]$.  
Since $\R$ is a local principal ideal ring, we may write $\mathfrak{m} = \langle \pi \rangle$.  
Then $R(t) = \pi S(t)$ for some $S(t) \in \R[t]$ with unit leading coefficient; this is possible since $\R$ is a local ring of length two.  
Moreover, $\deg R(t) = \deg S(t)$.  
Since $R \in N_A$, we have $\mathbf{0} = R(A) = \pi S(A)$, and hence $S(A) \in \M_n(\mathfrak{m})$.  
Applying $\theta$ gives $\overline{S}(\overline{A}) = \mathbf{0}$, which contradicts the minimality of $\deg \mathrm{Min}_{k}(A)$.
\end{proof}
\begin{lemma}\label{lem:min=char-regsem}
Let $A\in GL_n(\R)$ such that $\overline{A} \in GL_n(k)$ has irreducible characteristic polynomial. Then a minimal polynomial of $A$ over $\R$ is uniquely determined, which is the irreducible characteristic polynomial of $A$ over $\R$.
\end{lemma}
\begin{proof}
Let $F(t)\in \R[t]$ be the characteristic polynomial for $A$, which is monic. Then ${f}(t)$ will be the characteristic polynomial for $\overline{A}$ in $k[t]$, which is monic and irreducible. So, $F(t)$ must be monic irreducible in $\R[t]$, by \Cref{lem:f-irred-F-irred}.
Now as ${f}(t)$ is irreducible, so minimal polynomial of $\overline{A}$ let say $h(t)\in k[t]$ should be ${f}(t)$ itself. We can write ${f}(t) = h(t).1$. Note that $h(t)$ and $1$ are coprime. By Hensel's lemma version 1 \Cref{hensel-1}, there exist  $H(t), W(t)\in \R[t]$ such that $F(t) = H(t)W(t)$ with $\overline{H}(t)=h(t)$ and $\overline{W}(t)=1$. As $F(t)$ is monic irreducible, it guarantees that $H(t)=F(t)$ and $W(t)=1$, by \Cref{lem:f-irred-F-irred}. Note that $H(t)\in N_A$ as $W(A)$ is invertible.
    
Now, if possible, let there exist a monic polynomial $G(t)\in N_A$ such that $\deg(G(t))<\deg(H(t))$. Then ${g}(t)$ will be an annihilating polynomial of $\overline{A}$ in $k[t]$ such that $\deg({g}(t))< \deg(h(t))$, which leads to a contradiction. So, in this case $H(t)$ is the unique monic annihilating polynomial of $A$ of smallest degree, i.e, the minimal polynomial for $A$ over $\R$, which is nothing but the characteristic polynomial $F(t)$. 
\end{proof}
As a corollary, using \Cref{lem:tech-1-thm-1}, one obtains the following.
\begin{corollary}\label{cor:reg-sem-gen-null-ideal}
 Let $A\in GL_n(\R)$ such that $\overline{A} \in GL_n(k)$ is regular semisimple.
 Then a minimal polynomial of $A$ over $\R$ is uniquely determined, which is the irreducible characteristic polynomial of $A$ over $\R$.   
\end{corollary}
% We will show that in the particular case of ours, the null ideal is principal. We recall the following result due to Jacobson \cite[Theorem 2.14]{JacobsonAlgebraBook85}.
% \begin{lemma}\label{lem:jac}
% Let $\mathscr R$ be a commutative ring with unity. Let $f(t)$ and $g(t)\neq 0$ be polynomials in $\mathscr R[t]$ and let $m$ be the degree and $b_m$ be the leading coefficient of $g(t)$. Then, there exists an $n\in \mathbb{N}$ and polynomials $z(t)$ and $r(t)$ with either $r(t)=0$ or $\deg(r(t)) < \deg(g(t))$ such that 
% $$b_m^n f(t) = z(t)g(t) + r(t).$$
% \end{lemma}
Let $A\in \GL_n(\R)$ such that $\overline{A} \in \GL_n(k)$ has irreducible characteristic polynomial. 
Then, in this case, the minimal polynomial of $A$ over $\R$ is the unique monic polynomial of smallest degree that generates the null ideal $N_A$. 
However, this property does not hold in general. 
We provide two examples of such matrices,  one non-invertible and one invertible, whose minimal polynomial is not unique. 
Although this fact may be familiar to experts, we include these examples for the reader's convenience.

\begin{example}
Fix $\R = \mathbb{Z}/p^2\mathbb{Z}$ where $p$ is an odd prime. 
Then $\R$ is a local principal ideal ring of length two with unique maximal ideal $\langle p\rangle = p\R$. 
The nilpotency index of this maximal ideal is $2$. 
It is clear that the annihilator $\mathrm{Ann}_{\R}(p)= \langle p\rangle=p\R$. 
\begin{enumerate}
\item Take a non-invertible matrix, $B= \begin{pmatrix} 0 & p \\ 0 & 0 \end{pmatrix}\in \M_2(\R)$.
The corresponding null ideal is $N_B=\langle  t^2, \mathrm{Ann}_R(p)t \rangle = \langle t^2,pt\rangle$, (see \cite[Lemma 2.1]{Brown2005}) which is not principal in $\R[t]$. Now, if the minimal polynomial of $B$ is uniquely determined, then $N_B$ should be principal. So, for this $B$, a minimal polynomial is not uniquely determined. Note that the polynomials $F_1(t) = t^2$ and $F_2(t) = t^2+pt$ both are smallest degree monic annihilating polynomials for $B$ in $\R[t]$. So, they are the two distinct minimal polynomials for $B$.
\item Similarly, we can construct such an example for an invertible matrix. Let $D= \begin{pmatrix}  1 & p \\ 0 & 1 \end{pmatrix}\in \GL_2(\R)$. Note that $F_1(t)=t^2-2t+1$ and $F_2(t)= t^2+(p-2)t+(1-p)$ are two distinct minimal polynomials for $D$ in $\R[t]$.       
\end{enumerate}
\end{example}
A natural question arising from the example is the following: given a matrix $A \in \GL_n(\R)$, if $F(t)$ is any monic lift of $\mathrm{Min}_{k,\overline{A}}(t)$, must it always hold that $F(A)=\mathbf 0$? 
The answer is negative, as demonstrated by the following example.
\begin{example}
    Consider $\R=\Z/9\Z$ and $A=\begin{pmatrix}
        0 & -1\\
        1 & -1
    \end{pmatrix}\in\M_2(\Z/9\Z)$.
    Then $F(t)=t^2+t+4$ is a lift of $\mathrm{Min}_{\mathbb F_3,\overline{A}}(t)$, however $F(A)=\begin{pmatrix}
        3 & 0\\&3
    \end{pmatrix}\neq \mathbf 0$.
\end{example}
%%%%%%%%%%%%%%%%%%%%%%%%%%%%
\subsection{Conjugacy classes in $\GL_n(R)$} 

For a group $\mathscr G$, the image of a power map is invariant under the conjugation action (because $(ghg^{-1})^L = gh^Lg^{-1}$ for all $g, h\in\mathscr G$ and $L\in\Z$). Thus, to identify the elements of $\mathscr G$ which occur in the image of the power map, it is enough to identify those conjugacy classes of $\mathscr G$ lying in the image. A similar approach was used while dealing with finite general linear groups, orthogonal \& symplectic groups, and unitary groups, see \cite{KunduSinghGL}, \cite{PanjaSinghSymplectic25}, and \cite{PanjaSinghUnitary24} for the respective cases. When $R$ is a local ring of length $\ell$ (such as $\mathbb{Z}/p^\ell \mathbb{Z}$ or $\mathbb{F}_q[t]/\langle t^\ell\rangle$), describing the conjugacy classes in $\GL_n(R)$ becomes a difficult problem. According to Hill (see \cite{Hill95}), this difficulty arises because having such a description for all $\ell\geq 2$ is equivalent to classifying indecomposable $\widetilde{\mathscr{O}}_r$-lattices for all cyclic $p$-groups. This is a classification problem known to be wild (see \cite{GudivokPogorilyak89}). Here, $\mathscr{O}$ denotes the ring of integers of a $p$-adic field $K$, and $\widetilde{\mathscr{O}}$ is the ring of integers of $\widetilde{K}$, the maximal unramified extension of $K$ in an algebraic closure. We write $\widetilde{\mathscr{O}}_r = \widetilde{\mathscr{O}}/p^r \widetilde{\mathscr{O}}$. Nevertheless, there has been progress in addressing this problem over the years in the case $\ell = 2$; see \cite{Singla2010b, PrasadSinglaSpallone2015}. 

% Let $A\in \M_n(\R)$, $\theta(A)=\overline{A}\in \M_n(k)$, and $G_{\overline{A}}$ denote the centralizer of $\overline{A}$ inside $\GL_n(k)$. Then, we have the following (see \cite[Theorem 2.8]{PrasadSinglaSpallone2015}):

% \begin{lemma}\label{lem:sim-bij}
% There exists a bijection between the set of $G_{\overline{A}}$-orbits in the set of extensions $\ext_{k[t]}(M^{\overline{A}}, M^{\overline{A}})$ to the set of similarity classes of matrices in $\M_n(\R)$ that lie above the similarity class of $\overline A$ in $\GL_n(k)$.
% \end{lemma}
% The bijection in the lemma is by the assignment $\xi \mapsto A_\xi$, which can be found in \cite[Theorem 6.5]{PrasadSinglaSpallone2015}.

%%%%%%%%%%%%%%%%%%%%%%%%%%%%%%%%%%
\subsection{Conjugacy classes in the group $\mathrm{GL}_2(\Z/p^\ell\Z)$}\label{subsec:conj-gl}

In this subsection, $\R=\Z/p^2\Z$. We list down the similarity classes of matrices in $\GL_2(\R).$ There are four types of similarity classes $viz.$
\begin{enumerate}
\item $S(\alpha)  : \left( \begin{array}{cc} \alpha & 0 \\
0 & \alpha\end{array} \right);  \alpha\in \R^{\times}$
\item $D(\alpha,\delta,i)  :  \left( \begin{array}{cc} \alpha & 0 \\
0 & \delta\end{array} \right);   \alpha,\delta\in \R^{\times};\alpha-\delta\in p^iR^{\times}; 0\leq i<2$
\item $H(\alpha,\beta,i) :  \left( \begin{array}{cc} \alpha & p^{i+1}\beta \\
p^i & \alpha\end{array} \right); \alpha\in \R^{\times},\beta\in \R/p^{l-i-1}R; 0\leq i<2 $
\item $H'(\alpha,\beta,i)  :  \left( \begin{array}{cc} \alpha & p^i\epsilon\beta \\
p^i\beta & \alpha\end{array} \right); \alpha\in \R,\beta\in \R^{\times}, \alpha^2-\epsilon\beta^2p^{2i}\in \R^{\times}; 0\leq i<2; \epsilon$  is a fixed non square unit in $\R$.

\end{enumerate}

\section{Regular semisimple matrices as $L$-th power in $\mathrm{GL}_n(\R)$}\label{sec:separable}

An element $A\in \GL_n(\R)$ is said to be \emph{regular semisimple} if $\overline{A} \in\GL_n(k)$ is regular semisimple. This section deals with when a matrix of this kind in $\GL_n(\R)$ is a power. This problem over a field is dealt with in~\cite{KunduSinghGL}. We begin with showing that a regular semisimple element $A\in\GL_n(\R)$ has its characteristic polynomial $\Char_{\R, A}(t)$ a product of distinct fundamental irreducible polynomials from $\R[t]$. Then, similar to the field case, we show that such a matrix $A$ is in the image of $L$-power map if and only if each fundamental irreducible factor of $\Char_{\R, A}(t)$ is an $L$-power polynomial.
%%%%%%%%%%%%%%%%%%%%%%%%%%%%%
\subsection{Centralizer of a regular semisimple element in $\GL_n(\R)$}

Let $A \in \GL_n(\R)$ be a regular semisimple element. The centralizer of $\overline A \in \GL_n(k)$ is well known (see, for example, \cite[Theorem 2.3.11]{Singla2010}), and it is $k[\overline{A}] \cong \mathbb{F}_{q^n}$. Furthermore, viewing $\overline{A}$ as an element of $\M_n(\Fq)$, we have
\begin{align*}
\Zen_{\M_n(\Fq)}(\overline{A}) \cong \mathrm{End}_{\Fq[t]}(M^{\overline{A}}, M^{\overline{A}}).
\end{align*}
It is known that this isomorphism continues to hold when $\Fq$ is replaced by a commutative ring $R$ with unity. We include an explanation for this for the sake of completeness. 

Let $R$ be a commutative ring with unity. Let $A\in \M_n(R)$, and take $X\in \Zen_{M_n(R)}(A) $. Consider $M = R^n$ as an $R[t]$ module by $t\cdot v = Av$.
Define the map 
\begin{align*}
\phi_X\colon M^A\longrightarrow M^A\text{ by } v\mapsto X(v).
\end{align*}
Consider the map $$\Psi\colon \Zen_{M_n(R)}(A) \longrightarrow \mathrm{End}_{R[t]}(M^A, M^A)\text{ by } X\mapsto \phi_X.$$
Then, $\Psi$ is a ring homomorphism, with $\mathrm{Ker}(\Psi) = \left\{X\in \Zen_{M_n(R)}(A) \mid X(v)=0\text{ for all } v\in M^A\right\} =\left\{\mathbf 0\right\}.$ 
Hence, $\Psi$ is injective. 
Moreover, if we take any $T\in\mathrm{End}_{R[t]}(M^A, M^A)$, then from the action of $t$ on $M$ (which is by $A$) we have $T(t.v)=t.(Tv)$ which implies $T(Av)= A(Tv)$ for all $v\in M^A$. Hence $TA = AT$ and this implies $T\in \Zen_{\M_n(R)}(A) $, whence $\Psi$ is an isomorphism. Before going further, we note the following:
\begin{lemma}\label{lem:triv-hom}
Let $\R$ be a local principal ideal ring of length two. 
Consider the polynomial ring $\R[t]$, and let $I = \langle F(t) \rangle$ and $J = \langle G(t) \rangle$ be ideals in $\R[t]$, where $F(t)$ and $G(t)$ are monic fundamental irreducible polynomials such that $\gcd(\overline{F}(t), \overline{G}(t)) = 1$. Then, $\mathrm{Hom}_{\R[t]}(\R[t]/I,\R[t]/J)$ is zero.
\end{lemma}
\begin{proof}
Let $\varphi\in \mathrm{Hom}_{\R[t]}(\R[t]/I,\R[t]/J)$ such that $\varphi(1 + I)= v(t) + J$ for some $v(t) \in \R[t]$. Then, $\varphi(B(t) (H(t)+I))=B(t) \varphi(H(t) + I)$ for all $B(t)\in \R[t], H(t) \in \R[t]$. As $F(t) \in I$, $F(t) + I = F(t)(1 +I)=0+I$. This ensures that $\varphi(F(t)(1+I)) = 0+J$ which further implies $F(t)(v +J) = 0+J$ and hence $F(t)v(t)\in J$. Therefore, $F(t) v(t) = G(t) W(t)$ for some $w(t)\in R[t]$. As $\gcd(\overline{F}(t), \overline{H}(t))=1$, the ideals $I$ and $J$ are coprime in $\R[t]$; see  \cite[Section 2b]{Snapper1950}. So, there exist $\alpha(t), \beta(t) \in\R[t]$ such that $\alpha(t) F(t)+ \beta(t) G(t) =1$.
This implies $v(t) =v(t)(\alpha(t) F(t)+ \beta(t) G(t)) = \alpha(t) F(t) v(t) +\beta(t) G(t)v(t) = G(t)(\alpha(t) W(t) + \beta(t) v(t))$, and hence $v(t)\in J$. Thus, $\varphi(1 + I) = 0+J$. Therefore, $\varphi$ is the trivial homomorphism.
\end{proof} 

Now, if $\R$ is a finite local principal ideal ring of length two and $A\in \M_n(\R)$ is a regular semisimple matrix that lies over regular semisimple $\overline{A} \in M_n(k)$, we have the following: 
\begin{lemma}\label{lem:centra-reg-sem-R}
Let $\R$ be a finite local principal ideal ring of length two with its unique maximal ideal $\mathfrak{m}$. Let $A\in \M_n(\R)$ be a regular semisimple element with $\Char_{\R, A}(t) =\prod\limits_{i=1}^rF_i(t)$. Then, $\Zen_{\M_n(\R)}(A)\cong \bigoplus\limits_{i=1}^r \R[C_{F_i}]$ where $C_{F_i}$ is the companion matrix of the fundamental irreducible polynomial $F_i(t)$. Moreover, $|\Zen_{\GL_n(\R)}(A)|=|\mathfrak{m}|^n\prod\limits_{i=1}^r(q^{d_i}-1)$ where $d_i$ is degree of $F_i(t)$.
\end{lemma}
\begin{proof} 
Let $\Char_{\R, A}(t) = \prod\limits_{i=1}^r F_i(t)$, a factorization into monic fundamental irreducible polynomials, due to \Cref{lem:tech-1-thm-1}. From the discussion preceding this lemma, 
\begin{align*}
\mathrm{End}_{\R[t]}(M^A,M^A)\cong \bigoplus\limits_{i=1}^r \mathrm{End}_{\R[t]}(M^A_{F_i},M^A_{F_i}),
\end{align*}
where $M^A_{F_i} \cong \R[t]/\langle F_i(t)\rangle$ is the $i$-th primary component of $M^A$. The isomorphism above arises from $\R[t]/\langle F(t)\rangle\cong \bigoplus\limits_{i=1}^r\R[t]/\langle F_i(t)\rangle$ (by the Chinese Remainder Theorem), noting that for $i\neq j$ the polynomials $F_i(t), F_j(t)$ are coprime (and hence $\mathrm{Hom}_{\R[t]}\left(\R/\langle F_i(t) \rangle, \R/\langle F_j(t) \rangle\right)$ is trivial; see \Cref{lem:triv-hom}). Since $M^A_{F_i}\cong M^{C_{F_i}}$, where $C^{F_i}$ is the companion matrix of the fundamental irreducible polynomial $F_i(t)$, the first part of the result follows.

The isomorphism $\mathrm{Hom}_{\mathscr R}(\mathscr{R}/\mathscr{I},\mathscr R/\mathscr I)\cong \mathscr R/\mathscr I$ for any commutative ring $\mathscr R$ with unity and an ideal $\mathscr I$, is well-known. By setting $\mathscr R = \R[t]$ and $\mathscr I_i=\langle F_i(t)\rangle\subseteq \R[t]$, 
 $$\Zen_{\M_n(\R)}\left(C_{F_i}\right)\cong \mathrm{End}_{\R[t]}(M^{C_{F_i}},M^{C_{F_i}})
 =\mathrm{Hom}_{\R[t]}( \R[t]/\mathscr I_i, \R[t]/\mathscr I_i)\cong \R[t]/\mathscr I_i.$$
By \Cref{lem:null-ideal-prin}, it follows that  
$\R[t]/\langle F_i(t)\rangle \cong \R[A].$
Hence $\Zen_{\GL_n(\R)}(A)\cong (\R[A])^{\times}$.
Since $|\theta^{-1}(I_{d_i}) | =|\mathfrak{m}|^{d_i}(q^{d_i}-1)$, where $I_{d_i}\in\GL_{d_i}(k)$,
\begin{align*}
|\Zen_{GL_{d_i}(\R)}(C_{F_i})|=|\mathfrak{m}|^{d_i}(q^{d_i}-1).    
\end{align*}
It then follows that,
$$ |\Zen_{\GL_n(\R)}(A)| =\prod\limits_{i=1}^r |\Zen_{\GL_{d_i}(\R)}(C_{F_i})| 
 =\prod\limits_{i=1}^r(|\mathfrak{m}|^{d_i}(q^{d_i}-1))
 =|\mathfrak{m}|^n\prod\limits_{i=1}^r(q^{d_i}-1). 
$$
\end{proof}

%%%%%%%%%%%%%%%%%%%%%%%%%%%
\subsection{Understanding the image of $L$-th power map}

First, we begin with a statement for the roots of a polynomial similar to Hensel's lemma. 
\begin{proposition}\label{prop:hens-root-full}
Let $\R$ be a finite local principal ideal ring of length two with maximal ideal $\mathfrak m= \langle \pi \rangle$ and residue field $k$ of characteristic $p$. 
Let $A\in \M_n(\R)$ be regular semisimple, and let $F(t)\in \R[t]$ be monic of degree $d$. Suppose there exists $\widetilde{B} \in \M_n(k)$ such that $
\overline{F}(\widetilde{B}) = \overline{A}$ and 
$\overline{F}'(\widetilde{B})\in \GL_n(k)$ where $F'(t)$ is the formal derivative of $F(t)$ in $\R[t]$. 
Then there exists $B\in \GL_n(\R)$ with $\overline{B}=\widetilde{B}$ and $F(B)=A$.
\end{proposition}
\begin{proof}
Given $\widetilde{B} \in\GL_n(k)$ satisfy $\overline{F}(\widetilde{B})=\overline{A}$, we have $\widetilde{B}\in\Zen_{\M_n(k)}(\overline{A})$.
Since $\overline{A}$ is regular semisimple, $\Zen_{\M_n(k)}(\overline{A}) = k[\overline{A}]$, thus
$\widetilde{B} = \sum\limits_{i=0}^ra_i \overline{A}^i$ for some $a_i$. Now denote $\sum\limits_{i=0}^r b_i A^i = B_0 \in\R[A] \subseteq\M_n(\R)$, a lift of $\widetilde{B}$, under the map $\theta \colon \R[A] \longrightarrow k[\overline{A}]$, where $\theta(b_i)=a_i$. It can be proved that $\ker(\theta) = \mathfrak{m}[A]$.

Since $\overline{F(B_0)}=\overline{F}(\widetilde{B}) = \overline{A}$, there exists $C\in\R[A]$ such that $F(B_0)=A+ \pi C$ (by the use of $\theta$ map).
Note that this step is possible only because $\pi^2=0$.
As $\overline{F'(B_0)} = \overline{F}'(\widetilde{B})$ is invertible, ${F'(B_0)}^{-1} \in \R[A]^\times \subseteq\GL_n(\R)$. Set $D = -CF'(B_0)^{-1}$. Then, by the use of Taylor series expansion (see \cite{ConradHensel}), one gets that
\begin{align*}
    F(B_0+\pi D)= F(B_0)+ \pi DF'(B_0) = F(B_0)-\pi C=A,
\end{align*}
Since $B_0,\pi D\in\Zen_{\M_n(\R)}(A) = \R[A]\subseteq \M_n(\R)$ and $\pi^2=0$.
Further we have $\overline{B_0+\pi D}=\overline{B_0}=\widetilde{B}$.
\end{proof}
\begin{corollary}\label{prop-A-regsem-irred}
With the notation in the above \Cref{prop:hens-root-full}, suppose $A\in \GL_n(\R)$ is a regular semisimple matrix with its characteristic polynomial fundamental irreducible. Then, $A\in \im(\Phi_L)$ if and only if $\overline{A}\in \im(\overline{\Phi}_L)$.
\end{corollary}
\begin{proof}
Take the polynomial $F(t) = t^L$ in \Cref{prop:hens-root-full}. If $\widetilde{B}^L = \overline{A}$ for some $\widetilde{B} \in\GL_n(k)$, then $\overline{F}'(\widetilde{B}) = L\widetilde{B}^{L-1} \in \GL_n(k)$. Hence, by \Cref{prop:hens-root-full}, it follows that $\overline{A}\in\GL_n(k)$ implies that $A\in \GL_n(\R)$. For the other side, if $B\in\GL_n(\R)$ satisfy $B^L= A$, one has $\overline{B}^L=\overline{A}$.
\end{proof}
\begin{proposition}\label{prop:reg-sem-baby-case}
Let $\R$ be a finite principal ideal local ring of length two, with unique maximal ideal $\mathfrak m$, residue field $k$ of characteristic $p$ (an odd prime), and $L$ be a positive integer coprime to $p$. 
A regular semisimple element $A\in\GL_n(\R)$ with fundamental irreducible characteristic polynomial is an $L$-th power if and only if $\Char_{\R, A}(t)$ is an $L$-power polynomial.
\end{proposition}
\begin{proof}
For simplicity let us denote $\Char_{\R, A}(t)= \Char(t)$, with $\overline{\Char}(t) = g(t)$. Suppose $\Char(t^L)$ has a degree $n$ monic irreducible factor, say $V(t)$. Then, $\Char(t^L) = V(t)W(t)$, with $\overline{V}(t) = v(t)$  and $\overline{W}(t) = w(t)$.
Then, $\Char_{k,\overline{A}}(t) = g(t)$ is irreducible in $k[t]$. Clearly, ${g}(t^L) = v(t)w(t)$. Let $C_{v}$ be the companion matrix associated with $v(t)$. Then $g(C_{v}^L)= 0$.  As $g(t)$ is irreducible in $k[t]$, so $C_{v}^L$ is similar to $\overline{A}$, which implies that the characteristic polynomial of $C_{v}$ is irreducible in $k[t]$. So, $v(t)$ is an irreducible factor of $g(t^L)$ of degree $n$. Hence $\overline{A} \in \im(\overline{\Phi}_L)$, therefore by \Cref{prop-A-regsem-irred}, we get $A\in \im(\Phi_L)$.

Conversely, let us assume $A\in \im(\Phi_L)$. Then we get $\overline{A}\in \im(\overline{\Phi}_L)$. 
Hence there exist a degree $n$ monic irreducible factor of ${g}(t^L)$, say ${v}(t)$; see \cite[Theorem 1]{Otero90}. Since $g(t)$ is separable and $\gcd(L, p) = 1$, the polynomial $g(t^L)$ is also separable. Write $g(t^L)= v(t)w(t)$ for some $w(t)\in k[t]$. Note that $v(t)$ and $w(t)$ are mutually coprime polynomials in $k[t]$.
Then, by Hensel's lemma version 1, \Cref{hensel-1}, there exist unique $V(t), W(t)$ in $\R[t]$ such that $\Char(t^L) = V(t) W(t)$, where the coefficients of $v(t)$ and $w(t)$ are reduction modulo $\mathfrak{m}$ to the coefficients of $V(t)$ and $W(t)$ respectively. $V(t)$ should be irreducible of degree $n$, because otherwise $v(t)$ will be reducible, which is a contradiction. Hence, $V(t)$ is a degree $n$ monic irreducible factor of $\Char(t^L)$ in $\R[t]$.
\end{proof}

Now we proceed to prove the general case when the characteristic polynomial of $\overline{A}\in \GL_n(k)$ admits a factorization into distinct irreducible monic polynomials in $k[t]$. We need the following technical lemma before we prove the main result.  
\begin{lemma}\label{lem:tech-2-thm-1}
Let $B= \mathrm{diag}(B_1, B_2, \ldots, B_r)\in \GL_n(\R)$ be a block diagonal matrix such that $F_i(t) = \Char_{\R, B_i}(t)$ are fundamental irreducible polynomials in $\R[t]$ satisfying $\gcd(F_i, F_j)=1$ for all $1\leq i\neq j\leq r$. Further, let $\overline{B} \in\GL_n(k)$ be a regular semisimple element with characteristic polynomial $f(t) = {f}_1(t)\cdots {f}_r(t)$ such that $\overline{F_i}= f_i$.
Then, $B\in \im(\Phi_L)$ if and only if $B_i\in \im(\Phi_L)$ for all $1\leq i\leq r$.
\end{lemma}
\begin{proof}
Let $B_i\in \im(\Phi_L)$ be of size $n_i\times n_i$ where $n_i = \mathrm{deg}(F_i(t))$. So, there exists $D_i\in \GL_{n_i}(\R)$ such that $D_i^L = B_i$ for each $1\leq i\leq r$. Now, let us take 
$$ D = \begin{pmatrix}  % Uses parentheses ( )
D_1 & & & \\& D_2 & & \\ & & \ddots & \\ & & & D_r
\end{pmatrix}.$$
Clearly $D^L = B$. So, $B\in \im(\Phi_L)$.

Conversely, let $B\in \im(\Phi_L)$, so $\overline{B} \in \im(\overline{\Phi}_L)$. Hence, each $f_i(t)$ is an $L$-power polynomial, by \cite[Proposition 4.5]{KunduSinghGL}. Thus, $\overline{B_i} \in \im(\overline{\Phi}_L)$ for all $i$. The result then follows from \Cref{prop:reg-sem-baby-case}.
\end{proof}
\noindent Now we prove \Cref{thm:reg-sem-poly}, the main result of this section.
\theoremone*
% \begin{theorem}\label{thm:reg-sem-poly}
% Let $\R$ be a finite local principal ideal ring of length two, $\mathfrak{m}$ be its unique maximal ideal, and let $k$ be the residue field, having characteristic $p$, an odd prime.
% Given an integer $L>0$ coprime to $p$, a regular semisimple element $A\in \GL_n(\R)$ is an $L$-th power if and only if each $\R$-irreducible factor of $\Char_{\R}(A)$ is an $L$-power polynomial.
% \end{theorem}
\begin{proof}
Let $A\in \GL_n(\R)$ with characteristic polynomial $\Char(t)\in \R[t]$ be such that $\overline{A} \in\GL_n(k)$ is a regular semisimple element with characteristic polynomial $\overline{\Char}(t) = h_1(t)h_2(t)\cdots h_r(t)$. 
We have a corresponding factorization by virtue of \Cref{lem:tech-1-thm-1}, $\Char(t) = H_1(t)H_2(t)\cdots H_r(t)$ such that all $H_i(t)$ are monic irreducible in $\R[t]$ and $\overline{H_i}(t) = h_i(t)$ for each $1\leq i\leq r$. 
Hence, 
$$\overline{A} \sim_k \begin{pmatrix}  % Uses parentheses ( ) 
{C}_{h_1} & & & \\ & {C}_{h_2} & & \\ & & \ddots & \\ & & & {C}_{h_r} \end{pmatrix}$$
where $C_{h_i}$ is the companion matrix corresponding to $h_i(t)$. Now, let $M^A = \R[t]/\langle \Char(t)\rangle$ where $\Char(t)=H_1(t)H_2(t)\cdots H_r(t)$. Let $I_j=\langle H_j(t) \rangle$ in $\R[t]$. 
As $h_i(t)$ and $h_j(t)$ are coprime polynomials for $i\neq j$, using \Cref{lem:F-G-fund-irred}, $I_i$ and $I_j$ are comaximal ideals in $\R[t]$ for $i\neq j$. Now, by the Chinese Remainder Theorem 
$$\dfrac{\R[t]}{\langle F(t)\rangle } \cong \dfrac{\R[t]}{\langle H_1(t)\rangle}\bigoplus \dfrac{\R[t]}{\langle H_2(t)\rangle}\bigoplus \cdots \bigoplus \dfrac{\R[t]}{\langle H_r(t)\rangle}.$$ 
This further implies that 
$$A\sim_{\R} \begin{pmatrix}  % Uses parentheses ( )
C_{H_1} &  & & \\ & C_{H_2} & & \\ & & \ddots & \\ &   & & C_{H_r} \end{pmatrix}$$
where $C_{H_i}$ is the companion matrix corresponding to $H_i(t)$ that satisfies ${\overline C}_{H_i}=C_{h_i}$ for $1\leq i\leq r$.

Now, by \Cref{lem:tech-2-thm-1}, $A\in \im(\Phi_L)$ if and only if each blocks $C_{H_i}\in \im(\Phi_L)$ for $1\leq i\leq r$. Then, by \Cref{prop-A-regsem-irred} it follows that  $A\in \im(\Phi_L)$ if and only if  $H_i(t^L)$ has a degree $d_i$ irreducible factor in $\R[t]$ where $d_i=\deg(H_i(t))$, for each $1\leq i\leq r$. This completes the proof.
\end{proof}

\section{Compatible cyclic matrices as $L$-th powers in $\GL_n(\R)$}\label{sec:cyc}

We start with a few definitions. 
A matrix $A\in \GL_n(\R)$ is said to be \emph{cyclic} if $\overline{A}=\theta(A)\in\GL_n(k)$ is cyclic matrix.
Recall that an element $\overline{A}\in\GL_n(k)$ is said to be \emph{cyclic} if $\Char_{k,A}(t)=\mathrm{Min}_{k,A}(t)$; equivalently there exists $v\in k^n$ such that $k^n=k[t]\cdot v$, where $t\cdot v=Av$.
We first have the following result:
\begin{lemma}\label{lem:cyclic-down-up}
    Let $\R$ be a finite local ring of length two, and $A\in\GL_n(\R)$ be cyclic. 
    Then the $\R$-module $\R^n$ is a cyclic $\R[t]$-module; $t\cdot v=A\cdot v$.
\end{lemma}
\begin{proof}
Let $\widetilde{v}$ be a cyclic vector for the $k[t]$-module $k^n$ via the action $t\cdot \widetilde{v}=A\cdot\widetilde{v}$.
Let $w\in\R^n$ be a lift of $\widetilde{v}$.
Consider the set $\left\{w,Aw,\cdots,A^{n-1}w\right\}\subseteq \R^n$.
Since $\theta(A^iw)=\overline{A}^i\widetilde{v}$ and $\{\widetilde{v},\overline{A}\widetilde{v},\cdots,\overline{A}^{n-1}\widetilde{v}\}$ is a basis of the $k$-vector space $k^n$, by \cite[Proposition 2.8]{AtiyahMacBook69}, one gets that $\R[t]\cdot w=\R^n$.
\end{proof}
Using \Cref{lem:null-ideal-prin}, one sees that for a cyclic element $A\in\GL_n(\R)$ the null ideal is principal and satisfies 
$\Char_{\R,A}(t)=\mathrm{Min}_{\R,A}(t)$.  
Unlike the field case, however, the characteristic polynomial may fail to admit a factorization of the form  
$
\Char_{\R,A}(t)=\prod_{i=1}^\ell F_i(t)^{r_i},
$
where the $F_i$ are fundamental irreducible and pairwise coprime.
This is illustrated in the following example.
\begin{example}
    Consider the matrix $A=\begin{pmatrix}
    0&-1\\1&-1
\end{pmatrix}\in\GL_2(\Z/9\Z)$ which has $\Char_{\Z/9\Z,A}(t)=t^2+t+1$ to be irreducible.
The matrix $\overline{A}\in\GL_2(\mathbb{F}_3)$ cyclic, with $\Char_{\mathbb{F}_3,\overline{A}}(t)=(t-1)^2$.
Note that the characteristic polynomial of $A\in\GL_2(\R)$ is irreducible, and hence $\Char_{\R,A}(t)\neq(F(t))^r$ for some fundamental irreducible $F(t)\in \R[t]$.
\end{example}
Since our goal is to count cyclic elements in $\GL_n(\R)$ and in $\GL_n(\R)\cap\im(\Phi_L)$, we introduce the following definition for a special subclass of cyclic elements.
Given a cyclic matrix $\widetilde{A}\in\GL_n(k)$ we call $A\in\GL_n(\R)$ to be \emph{cyclic matrix compatible with $\widetilde{A}$} if (a) the polynomial $\Char_{\R,A}(t)$ has a factorization of the form $\prod\limits_{i=1}^{\ell}F_i(t)^{r_i}$ with $\langle F_i,F_j\rangle=\R$, (b) the polynomials $F_i$ are fundamental irreducible and $\Char_{k,\overline{A}}(t)=\Char_{k,\widetilde{A}}(t)=\prod\limits_{i=1}^\ell\overline{F_i}(t)^{r_i}$.
We call an element $A\in\GL_n(\R)$ a \emph{compatible cyclic element}, if there exists a cyclic matrix $\widetilde{A}\in\GL_n(k)$ such that $A$ is a cyclic matrix compatible with $\widetilde{A}$.
\subsection{Jordan canonical form for compatible cyclic matrices}
We proceed to develop a result that is `\emph{analogous}', in a certain sense, to a canonical form for cyclic $k[t]$-modules. 
It is well known from the theory of modules over a principal ideal domain that a matrix $A \in \M_n(\mathbb F_q)$ can be written (up to conjugacy) as a block matrix with blocks of the form $$J_r(f)=\begin{pmatrix}
C_f &  &   &   &   \\
 I & C_f &  &   &   \\
  &  I & \ddots &  &   \\
  &   & \ddots  & C_f &  \\
  &   &   & I  & C_f \\
\end{pmatrix}_{rd\times rd},$$
where $d$ is the degree of $f$, an irreducible factor of the characteristic polynomial of $A$, $C_f$ is the companion matrix of $f$, and $r$ is a positive integer; up to rearrangement of blocks, this canonical form is unique. 
When $A$ has the minimal and characteristic polynomial both to be $(h(t))^r$ for some positive integer $r$, where $h(t)\in \mathbb{F}_q[t]$ is a monic irreducible polynomial, then $A\sim_{\mathbb{F}_q} J_r(h)$.
As a corollary we get that, when $A$ has the minimal and characteristic polynomial both as $(h(t))^r\in \mathbb{F}_q[t]$ fore some positive integer $r$, then $A\sim_{\mathbb{F}_q} J_r(h)$.
The construction of the basis for this Jordan form depends on an isomorphism of two special rings.
\begin{lemma}\cite[Theorem 2.3.7]{Singla2010}\label{lem:iso-ring-thesis}
    Let $h(t) \in \mathbb{F}_q[t]$ be an irreducible polynomial of degree $d$
and let $E$ denote the field $\mathbb{F}_q[t]/h(t)$. Then the rings $\mathbb{F}_q[t]/(h(t)^r)$ and $E[u]/(u^r)$ are isomorphic.
are isomorphic.
\end{lemma}
The core part of the proof of \Cref{lem:iso-ring-thesis} depends on the following version of Hensel's lemma.
\begin{lemma}\label{lem:hensel-4}
 Let $h(t)$ be an irreducible polynomial over $\mathbb{F}_q[t]$.
Then, for each positive integer $r$, there exists $z_r(t) \in \mathbb{F}_q[t]$ such that $z_r(t) \equiv
t\pmod{h(t)}$, and $h(z_r(t)) \equiv 0 \pmod{h(t)^r}$.
\end{lemma}

We now turn to the construction of a basis that yields the Jordan form we aim to establish.
Define the map $\delta: \mathbb{F}_q[u,v]\longrightarrow \mathbb{F}_q[t],~\text{by }~u\mapsto h(t), ~v\mapsto z_r(t)$. 
Let $I=\langle h(v),u^r\rangle$ and $J=\langle h(t)^r\rangle.$ Then $\delta(I)=J$. 
By \Cref{lem:hensel-4}, we can write $t = z_r(t) + \phi_1(t) h(t)$.
Since both $z_r(t)$ and $h(t)$ lie in the image of the map $\delta$, it follows that $t$ is also in the image, and therefore $\delta$ is surjective.
Now $\delta$ is a ring homomorphism, because for any $g(u,v)$ in $\mathbb{F}_q[u,v]$ we have $\delta(g(u,v))=g(\delta(u),\delta(v))$.
% Now let us look into the following proposition
% \begin{proposition}{\color{blue}(prop.11.2)}
%     Let $R_1$ and $R_2$ be two commutative rings with unity and $\pi$ be a surjective ring homomorphism from $R_1$ to $R_2$. Let $I$ be an ideal of $R_1$ such that $\pi(I)=J$ be an ideal of $R_2$. Then $\pi$ induces a surjection $\widetilde{\pi}$ from $R_1/I$ to $R_2/J$.  
% \end{proposition}
% The induced map is $$\widetilde{\pi}:R_1/I\longrightarrow R_2/J$$ $$r_1+I\mapsto \pi(r_1)+J.$$
The map $\delta$ induces a surjection $\widetilde{\delta}:\mathbb{F}_q[u,v]/\langle h(v),u^r\rangle\longrightarrow \mathbb{F}_q[t]/\langle h(t)^r\rangle$ defined by $\bar{v}\mapsto \overline{z_r(t)},\bar{u}\mapsto \overline{h(t)}$.
Note that in the left-hand side ring $\Bar{v}$ means $v+I$ and in the right-hand side $\overline{z_r(t)}$ means $z_r(t)+J$.
Moreover $\widetilde{\delta}(\overline{g(u,v)})=\overline{g(\delta(u),\delta(v))}$ i.e. $\widetilde{\delta}(g(u,v)+I)=g(\delta(u),\delta(v))+J$.
Now, 
$\delta(g(u,v))+J=0+J$ implies $g(\delta(u),\delta(v))\in J$.
Next, $g(\delta(u),\delta(v))\in J$ implies that $g(h(t),z_r(t))\in J$. Therefore, by \Cref{lem:hensel-4}, there exists a function $\widetilde{\phi}_2\in \mathbb{F}_q[t]$ such that $g(h(t),z_r(t))=\widetilde{\phi}_2(t)\cdot h(t)^r$.
Now through the definition of the surjection $\delta$, there exists a $\phi_2(u,v)\in \mathbb{F}_q[u,v]$ such that $\phi_2(\delta(u),\delta(v))=\widetilde{\phi}_2(t)$. Hence we get $g(u,v)=\phi_2(u,v)u^r$. Therefore $g(u,v)\in \langle h(y),x^r\rangle=I$. Hence $g(\delta(u),\delta(v))\in J$ implies $g(u,v)\in I$, whence we have
$\ker (\widetilde{\delta})=\{g(u,v)+I\in \mathbb{F}_q[u,v]/\langle h(v),u^r\rangle \mid \delta(g(u,v))+J=0+J\}=0+I$.
Consequently $\widetilde{\delta}$ is an isomorphism from $\mathbb{F}_q[u,v]/\langle h(v),u^r\rangle$ to $ \mathbb{F}_q[t]/\langle h(t)^r\rangle$.

For a fixed $r$, define $\vartheta(t)=t-z_r(t)$ in $\mathbb{F}_q[t]$. 
Then by \Cref{lem:hensel-4} it follows that $\vartheta(t)\in \langle h(t)\rangle$. 
Moreover $\vartheta(t)\notin \langle h(t)^2\rangle$, see \cite[p. 17]{Singla2010}.
Hence $\vartheta(t)=\alpha(t)h(t)$ for some $\alpha(t)\in \mathbb{F}_q[t]$ such that $\overline{\alpha(t)}\in (\mathbb{F}_q[t]/\langle h(t)^r\rangle)^{\times} $.
Indeed, if it is not a unit in $\mathbb{F}_q[t]/\langle h(t)^r\rangle$, the element $\vartheta(t)$ must belong to $\langle h(t)^2\rangle$, which is a contradiction.

Also, $\vartheta(t)=t-z_r(t)$ implies $t=z_r(t)+\alpha(t)h(t)$. 
Therefore $\Bar{t}= \overline{z_r(t)}+\overline{\alpha(t)h(t)}$.
From here onward, we will denote $\alpha(t)$ as $\alpha$ in $\mathbb{F}_q[t]$.
Take the map $\mathbb{F}_q[t]/\langle h(t)^r\rangle\xrightarrow[]{\tilde{\delta}^{-1}} {\mathbb{F}_q[u,v]/\langle h(v),u^r\rangle}$
defined by $\bar{t}\mapsto \bar{\alpha'}\bar{u}+\bar{v}$, where $\bar{\alpha'}={\tilde{\delta}}^{-1}(\bar{\alpha})$ (here $\alpha'=\alpha'(u,v) \in \mathbb{F}_q[u,v]$).
Note that $\bar{\alpha'}$ must be a unit in $\mathbb{F}_q[u,v]/\langle h(v),u^r\rangle$.
The ring $\mathbb{F}_q[u,v]/\langle h(v),u^r\rangle$ is an $\mathbb{F}_q$ vector space with respect to the basis:
\begin{align*}
   \{1, \bar{v},\cdots,\bar{v}^{d-1},
  (\bar{\alpha'}\bar{u}),(\bar{\alpha'}\bar{u})\bar{v},\cdots,(\bar{\alpha'}\bar{u})\bar{v}^{d-1}, 
  \cdots,
  (\bar{\alpha'}\bar{u})^{r-1}, (\bar{\alpha'}\bar{u})^{r-1}\bar{v},\cdots,(\bar{\alpha'}\bar{u})^{r-1}\bar{v}^{d-1}\}
\end{align*}
Since the action of $A$ translates to an action by $\bar{t}$ with respect to the above basis, the matrix of multiplication by $\bar{\alpha'}\bar{u}+\bar{v}$ can be obtained from the association given as follows;
\begin{align*}
  &1\mapsto \bar{v} + \bar{\alpha'}\bar{u}\\
  &\bar{v}\mapsto \bar{v}^2 + (\bar{\alpha'}\bar{u})\bar{v}\\
  &\cdots\\
  &\bar{v}^{d-1}\mapsto \bar{v}^d+ (\bar{\alpha'}\bar{u})\bar{v}^{d-1}=-a_0-a_1\bar{v}-\cdots-a_{d-1}\bar{v}^{d-1}+ (\bar{\alpha'}\bar{u})\bar{v}^{d-1}
\end{align*}
which is $J_r(h)$; where $h(t)=a_0+a_1t+\cdots+a_{d-1}t^{d-1}+t^d \in \mathbb{F}_q[t]$.
The vector space $\mathbb{F}_q[t]/\langle h(t)^r\rangle$ can also be seen as an $\mathbb{F}_q$ vector space of the same dimension $rd$.
Given that $\tilde{\delta}$ is both an isomorphism and $\mathbb{F}_q$-linear, it sends a basis to a basis when the two rings are regarded as $\mathbb{F}_q$-vector spaces of equal dimension.
Hence applying $\tilde{\delta}$ on the above basis we obtain another basis 
\begin{align*}
\{&1, \overline{z_r(t)},\cdots,\overline{z_r(t)}^{d-1},\\&(\bar{\alpha}\overline{h(t)}),(\bar{\alpha}\overline{h(t)})\overline{z_r(t)},\cdots,(\bar{\alpha}\overline{h(t)})\overline{z_r(t)}^{d-1},\\
&\cdots,\\
&(\bar{\alpha}\overline{h(t)})^{r-1}, (\bar{\alpha}\overline{h(t)})^{r-1}\overline{z_r(t)},\cdots,(\bar{\alpha}\overline{h(t)})^{r-1}\overline{z_r(t)}^{d-1}\}    
\end{align*}
With respect to this basis, the matrix of multiplication by $\tilde{t}(=\overline{z_r(t)}+\overline{\alpha(t)h(t)})$ is $J_r(h)$, because \Cref{lem:hensel-4} gives $h(z_r(t))\equiv 0\pmod{h(t)^r}$ which implies $(z_r(t))^d\equiv-a_0-a_1z_r(t)-\cdots-a_{d-1}z_r(t)^{d-1} \pmod{h(t)^r}$.
We will denote this basis (last one) by $\mathscr{B}_{\mathbb{F}_q}$. We have the following result. 
\begin{proposition}\label{prop-canonical-form-type}
Let $\R$ be a finite local principal ideal ring of length two, with unique maximal ideal $\mathfrak{m}$ and the residue field $k\cong\Fq$ of odd characteristic. 
Let $A\in \GL_n(\R)$ be a compatible cyclic matrix with $\Char_{\R,A}(t)=(F(t))^r$, such that $\theta(A)=\overline{A}\in \GL_n(k)$ is a cyclic matrix with $\Char_{k,\overline{A}}(t)=(\overline{F}(t))^r$, where $F(t)\in\R[t]$ is a monic fundamental irreducible polynomial. Then $$ A\sim_{\R} J_{\R,F}(r)=\begin{pmatrix}
C_F &  &   &      \\
 I & C_F &  &      \\
  &   & \ddots   &   \\
  &   &   & C_F  \\
\end{pmatrix},$$ where $C_F$ is the companion matrix corresponding to the polynomial $F(t)\in \R[t]$, and $I$ is the identity matrix of size $\deg F\times \deg F$.
\end{proposition}

To prove this result, we need a Hensel-type lifting lemma per \Cref{lem:hensel-4}.
\begin{lemma}\label{lem-Hensel-type-5}
Let $\R$ be a finite local principal ideal ring of length two with a unique maximal ideal $\mathfrak{m}=\langle\pi\rangle$, and the residue field $k=\mathbb F_q$ of odd characteristic.
Let $F(t)\in\R[t]$ be a monic fundamental irreducible polynomial and $f(t)=\overline{F}(t)$.
Then, for every integer $r > 0$, letting $I_r = \langle F(t)^r \rangle \subseteq \R[t]$ and $\overline{I}_r = \langle f(t)^r \rangle \subseteq \mathbb{F}_q[t]$, the following holds: for every polynomial $z(t)$ in $\mathbb{F}_q[t]$ satisfying $z(t)\equiv t\pmod{f(t)}$ and $f(z(t))\in \overline{I}_r$, there exists a lift $Z(t)\in \mathscr{O}_2[t]$ of $z(t)$, such that $Z(t)\equiv t\pmod{F(t)}$ and $F(Z(t))\in I_r$.
\end{lemma}
We postpone the proof of \Cref{lem-Hensel-type-5} until after completing the proof of \Cref{prop-canonical-form-type}.
\begin{proof}[Proof of \Cref{prop-canonical-form-type}]
Consider the basis $\mathscr{B}_{\mathbb{F}_q}$, as above. 
Recall that $\vartheta(t)=\alpha h(t)$, and hence $z_r(t)-t+\alpha h(t)=0$. 
The following diagram
$$\begin{tikzcd}
	{\R[t]} && {k[t]} \\
	{\R[t]/\langle F(t)^r\rangle} && {k[t]/\langle f(t)^r\rangle}
	\arrow["{{\theta}}", from=1-1, to=1-3]
	\arrow["{\Gamma}", from=1-1, to=2-1]
	\arrow["{\gamma}", from=1-3, to=2-3]
	\arrow["{{\theta}^*}", from=2-1, to=2-3]
\end{tikzcd}$$
where the maps $\theta$, $\theta^*$, $\Gamma$ and $\gamma$ are the natural surjections, is commutative.

Consider $Z_r(t)\in \R[t]$, a lift of $z_r(t)\in k[t]$, such that $Z_r(t)\equiv t\pmod{F(t)}$ and $F(Z_r(t))\equiv 0\pmod{F(t)^r}$. 
The choice $Z_r(t)$ exists by \Cref{lem-Hensel-type-5}.
Then $Z_r(t)=t-\beta F(t)$ for some $\beta\in\R[t]$.
Since $Z_r(t)$ and $z_r(t)$ represent the same element in $\R[t]/\mathfrak{m}[t] = \mathbb{F}_q[t]$, which is a principal ideal domain, it follows that $\beta$ and $\alpha$ are equal as elements of $\R[t]/\mathfrak{m}[t]$. Consequently, $\beta \in \R[t]$ is a lift of $\alpha \in \mathbb{F}_q[t]$ via the map $\theta$.
    
Now, $\Gamma(t)=\Gamma(Z_r(t))+\Gamma(\beta F(t))$. 
The ring $\R[t]/\langle F(t)^r\rangle$ is a free $\R$ module. 
Consider $\langle\mathscr{B}_{\R}\rangle$, the $\R$-submodule of $\R[t]/\langle F(t)^r\rangle$ generated by $\mathscr{B}_{\R}$, where $\mathscr{B}_{\R}$ is the set 
\begin{align*}
  &\{1, \Gamma(Z_r(t)),\cdots,\Gamma(Z_r(t)^{d-1}),\\  
  &\Gamma(\beta F(t)),\Gamma(\beta F(t))\Gamma({Z_r}(t)),\cdots,\Gamma(\beta F(t))\Gamma({Z_r}(t)^{d-1}),\\
  &\cdots,\\
  &\Gamma((\beta{F(t)})^{r-1}), \Gamma((\beta{F(t)})^{r-1})\Gamma(({Z_r}(t)),\cdots,\Gamma((\beta{F(t)})^{r-1})\Gamma({Z_r}(t)^{d-1})\}.
\end{align*}    
Observe that $\theta^*(\Gamma(Z_r(t)))=\gamma(z_r(t))$; and $\theta^*(\Gamma(\beta F(t)))=\gamma(f(t))$. 
This further shows that $\theta^*(\mathscr{B}_{\R})=\mathscr{B}_{\Fq}$, adhering to the notations above.
As $\langle\mathscr{B}_{\Fq}\rangle\cong \Fq^n={\R}/{\mathfrak{m}}\otimes \R^n$, by Nakayama's Lemma \cite[Proposition 2.8]{AtiyahMacBook69}, we get $\langle\mathscr{B}_{\R}\rangle \cong \R^n$. 
Hence ${B}_{\R}$ is a minimal generating set for $\R[t]/\langle F(t)^r\rangle$ as a free $\R$ module.
As $F(\Gamma(Z_r(t)))=0$ in $\R[t]/\langle F(t)^r\rangle$, by \Cref{lem-Hensel-type-5}, with respect to this minimal generating set $\langle\mathscr{B}_{\R}\rangle$; the matrix of multiplication by $\Gamma(t)=\Gamma(Z_r(t)+\beta F(t))$ gives the form of matrix, as mentioned.
\end{proof}

\begin{proof}[Proof of \Cref{lem-Hensel-type-5}]
Consider the canonical surjection $\theta:\R[t]\longrightarrow \mathbb{F}_q[t]$, kernel of which is $\mathfrak{m}[t]$.
As $f(t)$ is irreducible in $\Fq[t]$, it must be separable. 
Therefore $\gcd(f(t),f'(t))=1$, whence there exist $a(t),b(t)\in\mathbb{F}_q[t]$ such that $a(t)f(t)+b(t)f'(t)=1$.
Let $A(t)$ and $B(t)$ be two arbitrary lifts of $a(t)$ and $b(t)$ respectively in $\R[t]$. 
As $A(t)F(t)+B(t)F'(t)\equiv a(t)f(t)+b(t)f'(t)\pmod{ \mathfrak{m}[t]}$, we can write $A(t)F(t)+B(t)F'(t)=1+\mathscr{M}(t)$ for some $\mathscr{M}(t)\in \mathfrak{m}[t]$.
As $\R$ is a local ring of length two, $1+\mathscr{M}(t)\in\R[t]$ is invertible. 
Taking $\widetilde{A}(t)=(1+\mathscr{M}(t))^{-1}A(t)$ and $\tilde{B}(t)=(1+\mathscr{M}(t))^{-1}B(t)$, one gets $\tilde{A}(t)f(t)+\tilde{B}(t)f'(t)=1$.
This gives $\tilde{B}(t)F'(t)\equiv1\pmod{F(t)}$,
i.e. $F'(t)$ is a unit in $\R[t]/I_1$.
Similarly, $f'(t)$ is a unit in $\Fq[t]/\overline{I}_1$.
Let us assume $F'(t)^{-1}\equiv D(t)\pmod{F(t)}$ and $f'(t)^{-1}\equiv d(t)\pmod{f(t)}$.
We will prove our result by induction on $r$.

\ul{We first deal with the case $r=1$}.
As $z(t)\equiv t\pmod{f(t)}$, we have $z(t)=t+f(t)c(t)$, for some $c(t)\in\mathbb{F}_q[t]$.
Take an arbitrary lift $C(t)\in\R[t]$ of $c(t)$, and define $Z(t)=t+F(t)C(t)$.
Then $Z(t)\equiv t\pmod{F(t)}$ and $Z(t)\equiv z(t)\pmod{\mathfrak{m}[t]}$.
Given a polynomial $H(t)=\sum_{i=0}^{d}{a_i}t^i\in\R[t]$, and two variables $u,v$, we can write 
$$ H(u+v)= a_0+\sum_{i=1}^{d}{(a_i((u^i+iu^{i-1}v)+\Gamma_i(u,v)v^2))},$$
where $\Gamma_i(u,v)\in\R[u,v]$ for all $i$.
Rearranging we have, $H(u+v)= H(u)+ H'(u)v+\Gamma(u,v)v^2$ where $\Gamma(u,v)=\sum_{i=1}^{d}{a_i\Gamma_i(u,v)}\in\R[u,v]$ and $H'(t)$ is the formal derivative of $H(t)$ (also see \cite[Example 2.2]{ConradHensel}). 
Since $Z(t)=t+F(t)C(t)$ and hence $F(Z(t))=F(t+F(t)C(t))=F(t)+F(t)C(t)F'(t)+(F(t)C(t))^2\Gamma(t,F(t)C(t))$, we get $F(Z(t))\equiv 0\pmod{F(t)}$. The statement is true for $r=1$. {This finishes the proof for the case $r=1$}.
Next, assume the statement is true for $r=2,3,\cdots,j-1$.

\ul{We now prove it for the case $r=j$.} 
As $f(t)^j|f(z(t))$ therefore $f(t)^{j-1}|f(z(t))$. 
Therefore $z(t)\equiv t\pmod{f(t)}$ and $f(z(t))\in \overline{I}_{j-1}$.
By the induction hypothesis, there exists a lift $V(t)\in \R[t]$ of $z(t)$ such that $V(t)\equiv t\pmod{F(t)}$ and $F(V(t))\in I_{j-1}$.
Therefore there exists $G(t)\in \R[t]$ such that $G(t)F(t)^{j-1}=F(V(t))$. 
Applying $\theta$, we get $g(t)f(t)^{j-1}=f(z(t))$.

We claim that $V(t)\equiv t\pmod{F(t)}$ implies that $F'(V(t))$ is a unit.
As $\tilde{A}(t)F(t)+\tilde{B}(t)F'(t)=1$ , replacing $t$ by $V(t)$, we have $\tilde{A}(V(t))F(V(t))+\tilde{B}(V(t))F'(V(t))=1$. 
As $V(t)\equiv t\pmod{F(t)}$, $F(V(t))\equiv F(t)\equiv 0\pmod{F(t)}$, whence 
$\tilde{B}(V(t))F'(V(t))\equiv1\pmod{F(t)}$. Therefore $F'(V(t))$ is a unit in $\R[t]/F(t)$.
Thus, we have proved our claim.

Writing $W_0(t)\equiv -G(t)(F'(V(t))^{-1}\pmod{F(t)}$, we get that there exists $W_0(t)\in \mathscr{O}_2[t]$ such that $G(t)+F'(V(t))W_0(t)\equiv 0\pmod{F(t)}.$
Applying $\theta$ and taking into consideration $f(t)^j|f(z(t))$ implies $f(t)|g(t)$ for $j>1$, we get $w_0(t)\equiv 0\pmod{f(t)}$. 
Therefore, we can write $W_0(t)=F(t)P(t)+N(t)$, where $P(t)\in \R[t]$ and $N(t)\in \mathfrak{m}[t]$.

Define $\widetilde{V}(t)=V(t)+F(t)^{j-1}W_0(t)=V(t)+F(t)^jP(t)+F(t)^{j-1}N(t)$. For $j\geq 2$, $\widetilde{V}(t)\equiv V(t)\equiv t\pmod{F(t)}$. 
Clearly $F(\widetilde{V}(t))=F(V(t)+F(t)^{j-1}W_0(t))\equiv F(V(t))+F(t)^{j-1}W_0(t)F'(V(t))\pmod{F(t)^j}$.
Replacing the value of $F(V(t))=G(t)F(t)^{j-1}$ here we get $$F(\widetilde{V}(t))\equiv F(t)^{j-1}[G(t)+F'(V(t))W_0(t)]\pmod{F(t)^j}.$$
Now by the construction $G(t)+F'(V(t))W_0(t)\equiv 0\pmod{F(t)},$ which further implies $F(\widetilde{V}(t))\equiv 0\pmod{F(t)^j}.$
We are almost done, but at this stage, it is not true that $\widetilde{V}(t)\equiv z(t)$, so we need to \emph{perturb} accordingly to reach the desired polynomial.

Applying $\theta$ on the polynomial $\widetilde{V}(t)$ we get, $$\widetilde{v}(t)=z(t)+f(t)^{j-1}w_0(t).$$ 
Therefore we have $\widetilde{v}(t)=z(t)+e(t)$ where $e(t)=f(t)^{j-1}w_0(t)$ is divisible by $f(t)^j$, since $w_0(t)\equiv0\pmod{f(t)}.$
Write $e(t)=f(t)^jw(t)\in\mathbb{F}_q[t]$. Take an arbitrary lift $W(t)\in\R[t]$ of $w(t)$ and, write $E(t)=F(t)^j W(t)$.
Now $\widetilde{V}(t)-E(t)=\widetilde{V}(t)+F(t)^j(-W(t)).$
Again, an application of $\theta$ on the polynomial $\widetilde{V}(t)-E(t)$ produces $\tilde{v}(t)-e(t)=z(t)+e(t)-e(t)=z(t)$. Let us define 
$$ Z(t)=\widetilde{V}(t)-E(t)=\widetilde{V}(t)+F(t)^j(-W(t)).$$
Therefore $Z(t)$ is a lift of $z(t)$ in $\mathscr{O}_2[t]$. Moreover, $Z(t)\equiv \widetilde{V}(t)\equiv t\pmod{F(t)}$ and $F(Z(t))\equiv F(\widetilde{V}(t))\equiv 0\pmod{F(t)^j}$. 
\ul{Hence, the induction step is complete}
\end{proof}
We emphasize that factorization in $\R[t]$ is not unique; therefore, in \Cref{prop-canonical-form-type}, no uniqueness is claimed for the choice of the polynomial $F$.
However, $F$ is a fundamental irreducible monic polynomial.
Thus, if $F(t)^r$ admits a factorization $F_1F_2\cdots F_r$ with each $F_i$ a fundamental irreducible monic polynomial, then each $F_i$ must necessarily be a lift of the same irreducible polynomial $\overline{F}(t)$ over $\Fq$. 
From \Cref{lem:F-G-fund-irred}, it is known that two polynomials in $\R[t]$ are relatively coprime if and only if their images under the canonical projection $\theta$ are coprime in $\Fq[t]$.
Hence, the factors $F_i$ cannot be coprime in $\R[t]$.
This implies that any alternative factorization of $F(t)^r$ in $\R[t]$ has no bearing on the $\R[t]$-module structure of $\R[t]/\langle F(t)^r\rangle$.
In particular, the module structure arising from any alternative factorization of $F(t)^r$ must coincide with the $\R[t]$-module structure on $\R[t]/\langle F(t)^r\rangle$ as discussed above. Consequently, any matrix obtained by choosing a different generating set will be conjugate to the form described in \Cref{lem-Hensel-type-5}.
\subsection{Candidacy  of a compatible cyclic element in the image of $L$-th power map}
We begin with the case of a compatible cyclic element whose characteristic polynomial has the form $F(t)^r$, where $F(t)\in\R[t]$ is a monic fundamental irreducible polynomial and $r>0$ is an integer.
\begin{proposition}\label{prop:cyclic-single-poly}
Let $\R$ be a finite local principal ideal ring of length two, with $\mathfrak{m}$ being its unique maximal ideal, and let $k$ be the residue field of characteristic $p$, an odd prime.  
Given an integer $L>0$ coprime to $p$, a compatible cyclic element $A\in \GL_n(\R)$ with $\Char_{\R,A}(t)=F(t)^r$ with $F(t)\in\R[t]$ being monic fundamental irreducible, is an $L$-th power if and only if $F(t)$ is an $L$-power polynomial. 
\end{proposition}
\begin{proof}
Let $\theta(F)=f$, and, $A\in \im(\Phi_L)$. 
By \Cref{prop-canonical-form-type}, 
$$A\sim_{\R} \begin{pmatrix}
C_F&  &   &      \\
I & C_F &  &      \\
  &   & \ddots   &   \\
  &   &   & C_F  \\
\end{pmatrix}\in\im(\Phi_L)\subseteq \GL_n(\R).$$
Applying $\theta$ we get that, $\overline{A}\in\GL_n(k)$, and by \cite[Proposition 4.2]{KunduSinghGL} $C_f\in \im(\overline{\Phi}_L)\subseteq \GL_n(k)$; whence $F$ is an $L$-power polynomial, by \Cref{thm:reg-sem-poly}.

Conversely, let $F$ be an $L$-power polynomial.
Therefore by \Cref{thm:reg-sem-poly}, $C_F\in \im(\Phi_L)$. Let $D\in \GL_d(\R)$ such that $D^L=C_F$.
Consider the matrix $B\in \GL_n(\R)$ defined as 
$$B= \begin{pmatrix}
D &  &   &      \\
 I & D &  &      \\
     & \ddots &  &   \\
     &   &   & D \\
\end{pmatrix}, $$
where $D$ appears $r$ many times as diagonal blocks. 
We claim that $B^L\sim_{\R}A$, which will prove the theorem when $\Char_{\R}(A)(t)=F(t)^r$.
Note that $$B^L=\begin{pmatrix} 
C_F &        &        &        &        \\ 
\binom{L}{1}D^{L-1} & C_F &        &        &        \\ 
\binom{L}{2}D^{L-2} & \binom{L}{1}D^{L-1} & \ddots &        &        \\ 
\vdots & \vdots & \ddots & C_F &        \\ 
\binom{L}{r-1}D^{L-r+1} & \binom{L}{r-2}D^{L-r+2} & \cdots & \binom{L}{1}D^{L-1} & C_F, 
\end{pmatrix}
$$ since $D^L=C_F$.
The matrix $B^L$ has characteristic polynomial $(F(t))^r\in \R[t]$ such that $\overline{B}^L\in \GL_n(k)$ has characteristic (and minimal) polynomial $f(t)^r\in k[t]$. 
We claim that the minimal polynomial of $B^L$ is also $F(t)^r$. 
It is enough to show that the minimal polynomial of $\overline{B}^L$ is $f(t)^r$.
 
\ul{The claim can be proved as follows}.
Consider the decomposition $\overline{B}^L= S+N$ where $$S=\mathrm{diag}(C_f,C_f,\cdots,C_f),\text{ and }N=\begin{pmatrix} 
\mathbf 0 &        &        &        &        \\ 
\binom{L}{1}\overline{D}^{L-1} & \mathbf 0 &        &        &        \\ 
\binom{L}{2}\overline{D}^{L-2} & \binom{L}{1}\overline{D}^{L-1} & \mathbf 0 &        &        \\ 
\vdots & \vdots & \ddots & \ddots &        \\ 
\binom{L}{r-1}\overline{D}^{L-r+1} & \binom{L}{r-2}\overline{D}^{L-r+2} & \cdots & \binom{L}{1}\overline{D}^{L-1} & \mathbf 0 
\end{pmatrix}$$
Note that the matrices $S$ and $N$ commute with each other as $\overline{D}C_f=C_f\overline{D}$. 
For the matrix $N^{r-1}$, its $(r,1)$-th block entry is $(LP^{L-1})^{r-1}$ and all other block entries are $\mathbf0_{d\times d}$.
Also $(L\overline{D}^{L-1})^{r-1}$ is invertible.
Therefore $N^{r-1}\neq \mathbf 0_{n\times n}$ and this implies the nilpotency index of $N$ is $r$. 
As $f(C_f)=\mathbf0_{d\times d}$, $f(S)=\mathbf 0_{n\times n}$.
As $f(t)$ is irreducible over a perfect field, it is separable, so $\gcd(f(t),f'(t))=1$. 
Hence there exists $a(t),b(t)$ in $\Fq[t]$ such that $a(t)f(t)+b(t)f'(t)=1$. 
Plugging $t=C_f$ instead of $t$ in the equation, one obtains $b(C_f)f'(C_f)=I$, and hence $f'(C_f)$ is invertible.
We conclude that the $(r,1)$-th block entry of $f(\overline{B}^L)^{r-1}$ is a non-zero block entry.
Therefore, $f(\overline{B}^L)^{r-1}$ is a nonzero matrix. 
As $\Char_{\Fq,{\overline{B}^L}}(t)=f(t)^r$ therefore this implies $f(\overline{B}^L)^r=\mathbf 0_{n\times n}$, and also, the $\mathrm{Min}_{\Fq}(\overline{B}^L)(t)=f(t)^r$; consequently the corresponding $\mathbb{F}_q[t]$ module $\mathbb{F}_q[t]/(h(t)^r)$ is cyclic. Hence the claim.
Therefore by \Cref{prop-canonical-form-type}, one obtains $$B^L\sim_{\R} \begin{pmatrix}
C_F &  &   &      \\
 I & C_F &  &      \\
  &   & \ddots &     \\
  &   &   & C_F 
\end{pmatrix}\sim_{\R} A.$$
Hence $A\in \im(\Phi_L)$.
\end{proof}
With all the necessary machinery in place, we are ready to prove \Cref{thm:cyc} in this section.
\theoremtwo*
\begin{proof}
To prove the theorem in generality, let $A\in \GL_n(\mathcal{O}_2)$ be a non-simple cyclic matrix with characteristic polynomial $F(t)=\prod\limits_{i=1}^\ell F_i(t)^{r_i}$ such that $\theta(A)= \overline{A}$ is a non-simple cyclic element in $\GL_n(k)$ with characteristic polynomial $f(t)=\prod\limits_{i=1}^\ell f_i(t)^{r_i}$; where $F_i$ are monic irreducible, relatively divisor-less polynomials in $\R[t]$ whose corresponding reductions $f_i$ in $k[t]$ are monic irreducible, mutually coprime polynomials.
    
First, let $A\in \im(\Phi_L)$. 
Then $\overline{A}\in \im(\overline{\Phi}_L)$. 
Consider the $\R[t]$ module $M^A$ associated to $A$, i.e. $\R[t]/\langle F(t)\rangle$.
By the Chinese Remainder Theorem, we can write $$M^A\cong\dfrac{\R[t]}{\langle F(t)\rangle}\cong \bigoplus\limits_{i=1}^\ell \dfrac{\R[t]}{\langle F_i(t)^{r_i}\rangle}.$$ 
Therefore $$A\sim_{\R} \begin{pmatrix}  % Uses parentheses ( )
A_1 &        &        & \\
& A_2    &        & \\
&        & \ddots & \\
&        &        & A_\ell
\end{pmatrix}$$
where each $A_i$ is non simple cyclic with corresponding characteristic polynomial $F_i(t)$ such that corresponding reduction matrix $\overline{A}_i$ are cyclic with characteristic polynomial $f_i(t)$ for each $1\leq i\leq\ell$.
Moreover $M^{\overline{A}}\cong \dfrac{k[t]}{\langle f(t)\rangle}\cong \bigoplus\limits_{i=1}^\ell \dfrac{k[t]}{\langle f_i(t)^{r_i}\rangle}.$ 
Now $\overline{A}\in \im(\overline{\Phi}_L)$ implies each $\overline{A}_i\in \im(\overline{\Phi}_L)$, by \cite[Theorem 1]{Otero90}.
Following the first part of the proof of the case $\ell=1$, we get that each $F_i(t)$ is an $L$-power polynomial in $\R[t]$.

Conversely, assume that for $1\leq i\leq \ell$, each $F_i(t)$ is an $L$-power polynomial in $\R[t]$. 
Therefore by \Cref{prop:cyclic-single-poly}, each $A_i\in \im(\Phi
_L)$.
Therefore there exists $B_i\in \GL_{r_id_i}(\R)$ such that $B_i^L=A_i$ for $1\leq i\leq\ell$. 
Set matrix $$B= \begin{pmatrix}  % Uses parentheses ( )
    B_1 &        &        & \\
        & B_2    &        & \\
        &        & \ddots & \\
        &        &        & B_\ell
\end{pmatrix}\in\GL_n(\R).$$ Then it satisfies $B^L= \begin{pmatrix}  % Uses parentheses ( )
    A_1 &        &        & \\
        & A_2    &        & \\
        &        & \ddots & \\
        &        &        & A_\ell
\end{pmatrix}$; whence $A\in \im(\Phi_L)$.
\end{proof}

\section{Probability generating functions for several classes of elements in $\GL_n(\R)$}\label{sec:counting}
This section is dedicated to the proof of \Cref{thm:gen-fun} and \Cref{thm:gen-fun-cyc}. 
We start by considering the lifts of elements $\overline{A} \in \GL_n(k) \cap \im(\overline{\Phi}_L)$, focusing particularly on those whose characteristic polynomial is irreducible, or which belong to either the regular semisimple or cyclic conjugacy classes.
We show that, in the above cases, if there exists a lift $A \in \GL_n(\R)$ of $\overline{A}$ such that $A \in \im(\Phi_L)$, then every lift of $\overline{A}$ lies in $\im(\Phi_L) \subseteq \GL_n(\R)$.
However, this property does not hold in general, as we demonstrate in the following example.
\begin{example}\label{exm:down-not-up}
Consider the ring $\R=\Z/9\Z$ and 
$
A=\begin{pmatrix}5&0\\0&2\end{pmatrix}\in\GL_2(\R), L=2.
$
We have 
\begin{align*}
\overline{\Phi}_2\!\left(\begin{pmatrix}0&1\\2&0\end{pmatrix}\right)=\overline{A}\in\GL_2(k),
\end{align*}
and
\begin{align*}
\overline{\Phi}_2^{-1}(\overline{A})
=\left\{
\begin{pmatrix}0&2\\1&0\end{pmatrix},
\begin{pmatrix}0&1\\2&0\end{pmatrix},
\begin{pmatrix}1&1\\1&2\end{pmatrix},
\begin{pmatrix}1&2\\2&2\end{pmatrix}
\right\}.
\end{align*}
If $B=\begin{pmatrix}\alpha&\beta\\\gamma&\delta\end{pmatrix}\in\GL_2(\R)$ satisfies $B^2=A$, then using $BA=AB$ one gets
\begin{align*}
\begin{pmatrix}5\alpha&2\beta\\5\gamma&2\delta\end{pmatrix}
=\begin{pmatrix}5\alpha&5\beta\\2\gamma&2\delta\end{pmatrix},
\end{align*}
which gives $3\beta=3\gamma=0$, so $\beta,\gamma\in\{0,3,6\}$.  
This contradicts $\overline{\beta},\overline{\gamma}\in\{1,2\}$.  
Hence, there is no $B\in\GL_2(\R)$ such that $B^2=A$.
\end{example}

First consider $\overline{A}\in \GL_n(k)$ such that $f(t)=\Char_{k}(\overline{A})(t)\in k[t]$ is an irreducible polynomial of degree $n$. 
If monic $F(t)\in \R[t]$ is such that $\theta(F)=f$, $F(t)$ must be irreducible in $R[t]$, by \Cref{lem:f-irred-F-irred}. 
Let $F(t)=t^n+a_{n-1}t^{n-1}+\cdots+a_1t+a_0$ and $f(t)= t^n+b_{n-1}t^{n-1}+\cdots+b_1t+b_0$ where $\theta({a_i})=b_i$ for $0\leq i\leq n-1$ and $b_0\neq 0$.
The number of choices for  each $a_i$ is $|\mathfrak{m}|$ (indeed for each fixed $b_i$, corresponding $a_i$ can be written as $b_i+m$ where $m\in \mathfrak{m}$). Thus, the total number of such irreducible monic polynomials $F(t)$ is $|\mathfrak{m}|^n$ and they are distinct.
Let these $|\mathfrak{m}|^n$ distinct polynomials in $\R[t]$ be $F_1(t),F_2(t),\cdots,F_{|\mathfrak{m}|^n}(t)$ and associate $A_1,A_2,\cdots,A_{|\mathfrak{m}|^n}\in\GL_n(\R)$ to each of the respective polynomials where $A_i=C_{F_i}$.
Then each of the $\R$-similarity classes $[A_1]_{\R},[A_2]_{\R},\cdots,[A_{|\mathfrak{m}|^n}]_{\R}$ lies over the $k$-similarity class of $[\overline{A}]_k$.
By \Cref{prop-A-regsem-irred}, these $\R$-similarity classes are also in $\im(\Phi_L)$; consequently all the monic lifts $F_i$ of $f$ is an $L$-power polynomial.
Hence, corresponding to each such conjugacy class of $\overline{A}\in\GL_n(k)$ (with irreducible characteristic polynomial) there exists $|\mathfrak{m}|^n$ distinct conjugacy classes in $\GL_n(R)\cap\im(\Phi_L)$.
Since the conjugacy classes of elements of $\GL_n(k)$ having an irreducible characteristic polynomial (of degree $n$) are determined by its characteristic polynomial, to count which of them belong to $\im(\overline{\Phi}_L)$, it is enough to count the number of $L$-power irreducible polynomials in $k[t]$.
This has been done in \cite[Proposition 3.3]{KunduSinghGL}. 
We recall that, if $N_{k,L}(q,d)$ denote the number of $L$-power polynomials of degree $d>1$ in $k=\Fq$, 
\begin{align*}
    N_{k,L}(q,d)=\frac{1}{d}\sum_{s|d}\mu(s)\frac{\gcd(L(q^{\frac{d}{s}}-1),(q^d-1))}{\gcd(L,q^s-1)},
\end{align*}
where $\mu$ is the M\"{o}bius function.
Since the $L$-power map $a\mapsto a^L$ is a homomorphism $\Fq^\times\longrightarrow\Fq^\times$, it follows that $N_{k,L}(q,1)=\dfrac{q-1}{(L,q-1)}$.
Hence, we have the following result concerning the count of monic fundamental irreducible polynomials in $\R[t]$.
\begin{lemma}\label{lem:count-fund-irred-L}
    Let $\R$ be a finite principal ideal ring of length two with unique maximal ideal $\mathfrak{m}$, and $N_{{\R}, L}(q,d)$ denote the number of $L$-power monic fundamental irreducible polynomials of degree $d$. Then
    \begin{align*}
        N_{{\R}, L}(q,d)=\begin{cases}
        \dfrac{q|\mathfrak{m}|-|\mathfrak{m}|}{\gcd(L,q-1)}&\text{ if }d=1\\
        \left(\dfrac{1}{d}\sum_{s|d}\mu(s)\dfrac{\gcd(L(q^{\frac{d}{s}}-1),(q^d-1))}{\gcd(L,q^s-1)}\right)|\mathfrak{m}|^d &\text{ if } d>1
        \end{cases},
    \end{align*}
    where $\mu$ is the M\"{o}bius function.
\end{lemma}
For example, if $\R= \mathbb{Z}/9\mathbb{Z}$, $d=2$, and $L=2$, one has $N_{\mathbb{F}_3,2}(3,2)=1$, and hence $3^2 N_{\mathbb{F}_3,2}(3,2)=9$. 
Therefore, up to conjugacy, there are $9$ classes in $\GL_2(\R)$ that lie in $\im(\Phi_2)$ above the conjugacy class (in $\GL_2(\mathbb{F}_3)$) corresponding to the polynomial $t^2+1\in \mathbb{F}_3[t]$.
The corresponding monic irreducible characteristic polynomials of these $9$ distinct classes are obtained using SageMath \cite{SageMath} and are listed below in \Cref{tab:irred-F-f}.

\begin{table}[h!]
\centering
\begin{tabular}{|c|c|c|}
\hline
{$F(t)$} & {Irreducibility in $\mathbb{Z}/9\mathbb{Z}[t]$} & $f(t)$ \\
\hline
$t^2 + 4$         & Irreducible & $t^2 + 1$ \\
$t^2 + 7$         & Irreducible & $t^2 + 1$ \\
$t^2 + 1$         & Irreducible & $t^2 + 1$ \\
$t^2 - 6t + 4$    & Irreducible & $t^2 + 1$ \\
$t^2 - 6t + 7$    & Irreducible & $t^2 + 1$ \\
$t^2 - 6t + 1$    & Irreducible & $t^2 + 1$ \\
$t^2 - 3t + 4$    & Irreducible & $t^2 + 1$ \\
$t^2 - 3t + 7$    & Irreducible & $t^2 + 1$ \\
$t^2 - 3t + 1$    & Irreducible & $t^2 + 1$ \\
\hline
\end{tabular}
\caption{\centering Irreducibility of polynomials in $\R[t]=\mathbb{Z}/9\mathbb{Z}[t]$ and their images in $\mathbb{F}_3[t]$}
\label{tab:irred-F-f}
\end{table}

We must mention that a monic irreducible polynomial need not be a fundamental irreducible polynomial.
Let $F(t)=t^2+t+1\in \R[t]$, where $\R=\Z/9\Z$.
The reduction is $\overline{F}(t)=(t-1)^2\in k[t]$, so $F(t)$ is not a fundamental irreducible polynomial.
However, $F(t)$ is irreducible in $\R[t]$.
If possible, let $F(t)=F_1(t)F_2(t)$ where $F_i(t)$ are both non-nilpotent non-units in $\R[t]$.
Then both are lifts of $t-1$. 
Then $F_i(t)=t-1+m_i(t)$ for some $m_i(t)\in\mathfrak{m}[t]=3\Z/9\Z[t]$.
Then $t^2+t+1=(t-1+m_1(t))(t-1+m_2(t))$ leads to $3t=(t-1)(m_1(t)+m_2(t))$, which is not possible (plug $t=1$).

Let $N(q,d)$ denote the number of monic irreducible polynomials (other than $t$) in $\Fq[t]$ of degree $d$. 
Then it is known that
\begin{align*}
    N(q,d)=\begin{cases}
        q-1&\text{if }d=1\\
        \dfrac{1}{d}\left(\sum\limits_{r|d}\mu(r)q^{\frac{d}{r}}\right)&\text{otherwise}
    \end{cases};
\end{align*}
see \cite[Lemma 1.3.10]{FulmanNeumannPraeger2005}.
\begin{proposition}\label{prop:gen-func-reg-sem-class}
    Let $\R$ be a finite local principal ideal ring of length two, with unique maximal ideal $\mathfrak{m}$ and residue field $\mathbb{F}_q$.  
Let $cs_n$ denote the number of conjugacy classes of regular semisimple elements in $\GL_n(\R)$, and let $cs_{n,L}$ denote the number of conjugacy classes of regular semisimple elements in $\GL_n(\R) \cap \im(\Phi_L)$, where $\gcd(L,p)=1$.  
Then the following equalities hold.

    \begin{equation}\label{eq:conj-reg-sem}
    1+\sum\limits_{n=1}^{\infty}cs_nz^n=\prod\limits_{d\geq 1} (1+z^d)^{|\mathfrak{m}|^d\cdot N(q,d)},
    \end{equation}
    where $N(q,d)$ denotes the number of irreducible polynomial of degree $d$ in $\Fq[t]$, and
    \begin{equation}\label{eq:conj-reg-sem-L}
   1+\sum\limits_{n=1}^{\infty}cs_{n,L}z^n=\prod\limits_{d\geq 1}(1+z^d)^{N_{\R,L}(q,d)}.
    \end{equation}
\end{proposition}
\begin{proof}
By definition, an element $A\in\GL_n(\R)$ is regular semisimple if $\overline{A}\in\GL_n(k)$ is regular semisimple.
Since $\overline{\Char_{\R,A}(t)}$ is separable (as $\overline{A}$ is regular semisimple), one gets that ${\Char_{\R,A}(t)}$ has a unique factorization into fundamental irreducible polynomials, using \Cref{lem:tech-1-thm-1}. 
It is enough to find the number of monic fundamental irreducible polynomials in $\R[t]$. 
Since $F(t)\in\R[t]$ is fundamental irreducible if it is irreducible in $\R[t]$ and $\theta(f)$ is irreducible in $\Fq[t]$, the number of monic fundamental irreducible polynomial in $\R[t]$ is $|\mathfrak{m}|^dN(q,d)$, using \Cref{lem:f-irred-F-irred}.
Hence \cref{eq:conj-reg-sem} follows.

Next, let $\Char_{\R,{A}}(t)=F(t)$, where $F(t)\in \R[t]$ is a monic polynomial such that $\overline{F}=f=\Char_{{k},\overline{A}}(t)$ is separable.
By \Cref{lem:tech-1-thm-1}, there exists a (unique) factorization of $F(t)$ in $\R[t]$ as $F(t)=F_1(t)F_2(t)\cdots F_r(t)$, where each $F_i(t)\in\R[t]$ is a monic fundamental irreducible polynomial, and $\overline{F}_i(t)=f_i(t)$. 
Furthermore, by \Cref{thm:reg-sem-poly}, $A\in\GL_n(\R)\cap\im(\Phi_L)$ if and only if each $F_i$ is an $L$-power polynomial. 
Then \cref{eq:conj-reg-sem-L} follows from \Cref{lem:count-fund-irred-L}
\end{proof}
We conclude this section with the proof of \Cref{thm:gen-fun} and \Cref{thm:gen-fun-cyc}.
The two results in the first theorem are analogous to the formulae in \cite[Section 2]{Wall1999Counting} and \cite[Theorem 5.3(1)]{KunduSinghGL}
\theoremthree*
\begin{proof}
    We use \Cref{prop:gen-func-reg-sem-class} here. 
    For a given regular semisimple element $A\in\GL_d(\R)$, with fundamental irreducible characteristic polynomial, by \Cref{lem:centra-reg-sem-R}, $|\Zen_{\GL_d(\R)}(A)|=|\mathfrak{m}|^d(q^d-1)$.
    This proves both of the equalities above.
\end{proof}
To prove \Cref{thm:gen-fun-cyc}, it is necessary to determine the size of the centralizer of a compatible cyclic element.  
For this, we first describe the structure of its centralizer algebra.  
The argument parallels that of \Cref{lem:centra-reg-sem-R}, but we include the details here for completeness.

\begin{lemma}\label{lem:centra-comp-cyc-R}
Let $\R$ be a finite local principal ideal ring of length two with its unique maximal ideal $\mathfrak{m}$. Let $A\in \M_n(\R)$ be a compatible cyclic element with $\Char_{\R, A}(t) =\prod\limits_{i=1}^\ell(F_i(t))^{r_i}$. Then, $\Zen_{\M_n(\R)}(A)\cong \bigoplus\limits_{i=1}^\ell \R[J_{\R,F_i}(r_i)]$ where $J_{\R,F_i}(r_i)$ are as in \Cref{prop-canonical-form-type}.
Moreover, if $F(t)\in \R[t]$ is a fundamental monic irreducible polynomial of degree $d$, $\left|\Zen_{\GL_{dr}(\R)}\left(J_{\R,F}(r)\right)\right|=|\mathfrak{m}|^{dr}q^{d(r-1)}(q^d-1)$.
\end{lemma}
\begin{proof}
We will prove the result for the case $\ell=1$.
Let $A\in GL_n(\R)$ be a compatible cyclic matrix, $\Char_{\R,A}(t)=F(t)^r$, where $F(t)$ is monic fundamental irreducible in $\mathcal{O}_2[t]$ of degree $d$ with reduction $\overline{F}(t)=f(t)$.
By \Cref{lem:tech-1-thm-1}, the null ideal $N_A=\langle F(t)^r\rangle$.
From the isomorphism $\Zen_{\M_n(\R)}\left(A\right)\cong \mathrm{End}_{\R[t]}(M^{A},M^{A})
 =\mathrm{Hom}_{\R[t]}( \R[t]/N_A, \R[t]/N_A)\cong \R[t]/N_A$, $N_A=\langle(F(t)^r)\rangle$, and \Cref{prop-canonical-form-type}, it follows that $\Zen_{\M_n(\R)}(A)= \R[A]\cong \R[J_{\R,F}(r)]$.

It remains to prove the second part of the statement.
In this regard, note that $\Zen_{\GL_{dr}(\R)}(A)=\R[A]^\times$.
Consider the following exact sequence of groups
$$ \begin{tikzcd}
	\mathbf0 & {\ker(\theta)} & {(\R[A])^{\times}} & {(k[\Bar{A}])^{\times}} & \mathbf 0
	\arrow[from=1-1, to=1-2]
	\arrow[from=1-2, to=1-3]
	\arrow["{\theta}", from=1-3, to=1-4]
	\arrow[from=1-4, to=1-5]
\end{tikzcd},$$
induced by $\theta$.
Then $$\ker\theta=\left\{\sum\limits_{i=0}^{rd-1}a_iA^i:a_i\in\mathfrak{m}\right\},$$ which further proves that $|\Zen_{\GL_n(\R)}(A)|=|\mathfrak{m}|^{rd}|(k[\Bar{A}])^{\times}|$.
\end{proof}
The following theorem may be viewed as an analogue of \cite[Theorem 5.3(2)]{KunduSinghGL}.
\theoremfour*
\begin{proof}
By definition, an element $A\in\GL_n(\R)$ is called compatible cyclic if $\overline{A}\in\GL_n(k)$ is cyclic, and the characteristic polynomials satisfy  
\[
\Char_{\R,A}(t)=\prod_{i=1}^{\ell}F_i(t)^{r_i}, 
\qquad 
\Char_{k,\overline{A}}(t)=\prod_{i=1}^{\ell}\overline{F}_i(t)^{r_i},
\]
where each $F_i$ is a monic fundamental irreducible polynomial, and $F_i$ and $F_j$ are coprime whenever $i\neq j$.

By \Cref{lem:snap-factor}, a factorization of $\Char_{\R,A}(t)$, as above, is unique. 
Since $F_i(t)\in\R[t]$ is fundamental irreducible if it is irreducible in $\R[t]$ and $\theta(F_i)$ is irreducible in $\Fq[t]$, the number of monic fundamental irreducible polynomial in $\R[t]$ is $|\mathfrak{m}|^dN(q,d)$, using \Cref{lem:f-irred-F-irred}.

We are given with $\Char_{\R,A}(t)=\prod\limits_{i=1}^{\ell}F_i(t)^{r_i}$, where each $F_i(t)\in\R[t]$ is a monic fundamental irreducible polynomial, and $\overline{F}_i(t)=f_i(t)$. 
Furthermore, by \Cref{thm:cyc}, $A\in\GL_n(\R)\cap\im(\Phi_L)$ if and only if each $F_i$ is an $L$-power polynomial. 

Let $cr_n$ denote the number of conjugacy classes of compatible cyclic elements in $\GL_n(\R)$, and let $cr_{n,L}$ denote the number of conjugacy classes of compatible cyclic elements in $\GL_n(\R) \cap \im(\Phi_L)$, where $\gcd(L,p)=1$.  Hence because of the above discussion,

    \begin{equation*}\label{eq:conj-comp-cyc}
    1+\sum\limits_{n=1}^{\infty}cr_nz^n=\prod\limits_{d\geq 1} (1-z^d)^{-|\mathfrak{m}|^d\cdot N(q,d)},
    \end{equation*}
    where $N(q,d)$ denotes the number of irreducible polynomial of degree $d$ in $\Fq[t]$, and
    \begin{equation*}\label{eq:conj-comp-cyc-L}
   1+\sum\limits_{n=1}^{\infty}cr_{n,L}z^n=\prod\limits_{d\geq 1}(1-z^d)^{-N_{\R,L}(q,d)}.
   \end{equation*}
Then the result follows in view of \Cref{lem:centra-comp-cyc-R}.
\end{proof}
\begin{remark}
A more general version of \Cref{lem:centra-comp-cyc-R} holds: for a cyclic matrix $A\in\GL_n(\R)$, one has $\Zen_{\M_n(\R)}(A)=\R[A]$.  
Using this, one can show that an element $A\in\GL_n(\R)$ lies in $\im(\Phi_L)$ if and only if $\overline{A}\in\GL_n(k)\cap\im(\overline{\Phi}_L)$; the proof being a verbatim replica of \Cref{prop:hens-root-full}.  
However, counting such elements is more difficult since we lack a factorization analogous to the compatible cyclic case.
\end{remark}
\section{Roots in $\GL_n(\R)$, roots in $\GL_n(k)$ and the hypothesis $\gcd(L,p)=1$}\label{sec:res-cond-L-p}
It is known \cite[Proposition 4.5]{KunduSinghGL} that a regular semisimple element $X \in \GL_n(\Fq)$ is an $L$-th power if and only if all the irreducible factors of its characteristic polynomial $\Char_{\Fq}(X)(t) \in \Fq[t]$ are $L$-power polynomials.
We have proved in \Cref{thm:reg-sem-poly} that an analogous statement holds for regular semisimple elements of $\GL_n(\R)$.
This naturally leads to the question: if $A \in \GL_n(\R)$ is such that $\theta(A) \in \GL_n(\Fq) \cap \im(\overline{\Phi}_L)$, must it follow that $A \in \im(\Phi_L)$?
The following example shows that this need not hold in general.
% In \Cref{thm:reg-sem-poly} and \Cref{thm:cyc}, we imposed the condition $\gcd(L, p) = 1$.
% A natural question is whether the results continue to hold if this restriction is removed.
% In this section, we provide examples that demonstrate the failure of both theorems in the absence of this crucial assumption.
% Thus, although the condition may initially appear artificial, it is in fact an essential part of the argument and cannot be omitted. 
% Furthermore, constructing a counterexample to \Cref{thm:reg-sem-poly} in the case $\gcd(L, p) \ne 1$ also yields a counterexample to \Cref{thm:cyc}.

% \Cref{thm:reg-sem-poly} can be restated as follows: under its hypotheses, an element $A \in \GL_n(\R)$ is an $L$-th power if and only if $\overline{A} \in \GL_n(k)$ is an $L$-th power.
% It is clear that if $A$ has an $L$-th root, then so does $\overline{A}$.
% \Cref{thm:reg-sem-poly} establishes the converse implication under the assumption $\gcd(L, p) = 1$.
% In what follows, we provide an example to demonstrate how the converse fails when this condition is dropped.

\begin{example}
We work with the ring $\R=\mathbb{Z}/9\mathbb{Z}$, $k\cong\mathbb{F}_3$, $L=3$ and $n=2$, so that $\gcd(L,p)=3\neq 1$. 
Consider the matrix $A= \begin{pmatrix}
 3 & 1 \\
5 & 0
\end{pmatrix} \in \GL_2(\mathbb{Z}/9\mathbb{Z})$. 
Applying the reduction map $\theta$ yields $\overline{A}= \begin{pmatrix}
 0 & 1 \\
2 & 0
\end{pmatrix}\in \GL_2(k) $.
We have $\Char_{\R,A}(t)=F(t)=t^2-3t-5\in \R[t]$, and, $\Char_{k,\overline{A}}(t)=f(t)=t^2+1\in k[t]$. 
The polynomials $F(t)$ and $f(t)$ are irreducible in $\R[t]$ and $k[t]$, respectively. 
Now, $\overline{A}^3=\begin{pmatrix}
    0 & 2\\ 1 & 0
\end{pmatrix}\sim_{k}\overline{A}$, and hence $\overline{A}\in \im(\overline{\Phi}_3)$. 
We claim that there does not exist any $B\in \GL_2(\R)$ such that $B^3=A$.

If possible, let there be $B\in \GL_2(\R)$ such that $B^3=A$. It can be shown that the order of $A$ is $\textrm{ord}(A)=12$, and hence
 $B^L=A$ implies $\textrm{ord}(B^L)=12$. 
 Hence $\dfrac{\textrm{ord}(B)}{\gcd(L,\textrm{ord}(B))}=12$ and further $ \textrm{ord}(B)=12\gcd(L,\textrm{ord}(B))$. 
Similarly, $\textrm{ord}(\overline{A})=4$. 
Therefore we have $\textrm{ord}(\overline{B}^L)=4$ which implies $ \textrm{ord}(\overline{B})=4\gcd(L,\textrm{ord}(\overline{B}))=4\ell$ where $\ell= \gcd(L,\textrm{ord}(\overline{B}))$.
Now $|\GL_2(\mathbb{F}_3)|= (3^2-1)(3^2-3)=48$. 
As $\textrm{ord}(\overline{B})|48$ and $\textrm{ord}(\overline{B})$ is a multiple of $4$, $\textrm{ord}(\overline{B})\in\{4,8,12,16,24,48\}$.
As $\GL_2(\mathbb{F}_3)$ is non-abelian, $\textrm{ord}(\overline{B})\neq 48$. 
Thus $\ell\in\{1,2,3,4,6\}$.
% We check the values of $\ell$ based on possible values of $\textrm{ord}(\overline{B})$ as follows;
% $$ \textrm{ord}(\overline{B})=4\implies \ell=1$$ $$\textrm{ord}(\overline{B})=8\implies \ell=2$$
% $$\textrm{ord}(\overline{B})=12\implies \ell=3$$ $$\textrm{ord}(\overline{B})= 16\implies \ell=4$$
% $$\textrm{ord}(\overline{B})=24\implies \ell=6.$$
If $\textrm{ord}(\overline{B})\in\{8,16,24\}$, $L$ must be even, which is not possible; $L$ being $3$.
Hence  $\textrm{ord}(\overline{B})\in\{4,12\}$. 
Now we investigate the possible value(s) of the order of $B$.

Let $\textrm{ord}(B)=s$ then $B^s=I$. That implies $\overline{B}^s=\overline{I}.$ When $\textrm{ord}(\overline{B})=4$, $4|s$ implies $s=4y$ for some integer $y$. 
From the equality $\textrm{ord}(B)=12\gcd(L,\textrm{ord}(B))$, plugging the values of $\textrm{ord}(B)$ and $L$, we get $y=3\gcd(3,4y)$. 
If $\gcd(3,4y)=1$ then $y=3$ which cannot be possible because then $\gcd(3,4y)\neq1$. 
As $\gcd(3,4y)|3$ and $3$ is a prime number, the only possibility is $\gcd(3,4y)=3$. 
This implies $y=9$, and $\textrm{ord}(B)=s=36$.

When $\textrm{ord}(\overline{B})=12$ then $s=12z$ for some integer $z$. 
The equality $\textrm{ord}(B)=12\gcd(L,\textrm{ord}(B))$ implies $z=\gcd(3,12z)=3$, and so $\textrm{ord}(B)=s=36$.

So, if such a $B$ exists, $\textrm{ord}(B)$ should be $36$. 
As $B^3=A$, $B\in \Zen_{\GL_2(\R)}(A)$. 
So the cyclic group $\langle B\rangle$ generated by $B$ is a subgroup of order $36\subseteq\Zen_{\GL_2(\R)}(A)$. 
It is well known that $|\Zen_{\GL_2(\R)}(A)|=\dfrac{|\GL_2(\R)|}{|[A]_{\R}|}$, where $|[A]_{\R}|$ is the size of the conjugacy class of $A$ in $\GL_2(\R).$  
Decompose the matrix $A$ as 
$$A=\begin{pmatrix}
 3 & 1 \\
5 & 0
\end{pmatrix}=\begin{pmatrix}
 6 & 0 \\
0 & 6
\end{pmatrix}+\begin{pmatrix}
 -3 & 1 \\
5 & 3
\end{pmatrix}.$$ 
Since $$\begin{pmatrix}
 1 & -3 \\
0 & 5
\end{pmatrix}^{-1}\begin{pmatrix}
 -3 & 1 \\
5 & 3
\end{pmatrix}\begin{pmatrix}
 1 & -3 \\
0 & 5
\end{pmatrix}=\begin{pmatrix}
 0 & 5 \\
1 & 0
\end{pmatrix},$$ one has $A\sim_{\R} \begin{pmatrix}
 6 & 5 \\
1 & 6
\end{pmatrix}$.
Now $A'=\begin{pmatrix}
 6 & 5 \\
1 & 6
\end{pmatrix}$ is a member of $H'(\alpha,\beta,i)$ family. Here in particular for $A'$ we have $\alpha=6$, $\epsilon=5$, $i=0$, $\beta=1$; see \Cref{subsec:conj-gl}.
So, the size of conjugacy class of $A'$ in $\GL_2(\R)$ is $|[A']_{\R}|=(3-1)3^{4-1}=54$, \cite[p. 1291]{BarringtonCliffWen2010}, which is $|[A]_{\R}|$ as well.
Now let us calculate $|\GL_2(\R)|$.
Since $|\GL_2(\R)|=3888$, we get $|\Zen_{\GL_2(\R)}(A)|=\dfrac{|\GL_2(\R)|}{|[A]_{\R}|}=\dfrac{3888}{54}=72$.

From previous discussion, $\langle B\rangle \subseteq \Zen_{\GL_2(\R)}(A) $ and $|\langle B\rangle |=36$, which implies $\langle B\rangle $ is an index $2$ subgroup of $\Zen_{\GL_2(\R)}(A)$; hence normal in $\Zen_{\GL_2(\R)}(A)$. 
Let us consider the following action by conjugation
$\Zen_{\GL_2(\R)}(A) \times \langle B\rangle \longrightarrow \langle B\rangle$, by $(g,B)\mapsto g.B=gBg^{-1}.$
Take $g\in\Zen_{\GL_2(\R)}(A)$ and assume $gBg^{-1}=B^j$ for some $j$.
This gives $\textrm{ord}(B^j)=\textrm{ord}(gBg^{-1})=\textrm{ord}(B)$, and hence 
$\dfrac{\textrm{ord}(B)}{\gcd(j,\textrm{ord}(B))}=\textrm{ord}(B).$ Hence $\gcd(j,\textrm{ord}(B))=1.$
This happens for any $g\in \Zen_{\GL_2(\R)}(A)$.
Therefore the orbit of $B$ is $\Zen_{\GL_2(\R)}(A).B=\left\{B^j| \gcd(j,\textrm{ord}(B))=1\right\}$. 
Let us denote this by $\mathrm{Orbit}(B)$. 
So, $|\mathrm{Orbit}(B)|=\phi(36)=12$, where $\phi$ is the Euler's totient function.
Then, the Orbit-Stabilizer theorem immediately gives $|\mathrm{Stab}_{\Zen_{\GL_2(\R)}(A)}(B)|=\dfrac{|\Zen_{\GL_2(\R)}(A)|}{|\mathrm{Orbit}(B)|}=6$.

Note that $|\R^{\times}|=6$ and all the $\lambda I$'s such that $\lambda\in R^{\times}$ are in $\mathrm{Stab}_{\Zen_{\GL_2(\R)}(A)}(B)$. Moreover $B$ itself, which should not be of the form $\lambda I$, is also a member of $\mathrm{Stab}_{\Zen_{\GL_2(\R)}(A)}(B)$.
As $\mathrm{Stab}_{\Zen_{\GL_2(\R)}(A)}(B)\subseteq \Zen_{\GL_2(\R)}(A) $, therefore $|\mathrm{Stab}_{\Zen_{\GL_2(\R)}(A)}(B)|\geq 7$; which is a contradiction. So no such $B$ exists in $\GL_2(\R)$ such that $B^3=A$. Hence $A\not\in \im(\Phi_3)$. 
Hence, we have shown that the \ul{existence of an $L$-th root of $\overline{A}$ need not guarantee the existence of an $L$-th root of $A$ in $\GL_2(\R)$}. 
This establishes the claims we made in the beginning of the example.
\end{example}
Surprisingly, this phenomenon is reflected in the irreducible factors of $F(t^3)$, as we shall demonstrate below.

We check whether $F(t^3)$ has any monic irreducible factor of degree $2$ in $\R[t]$.
Suppose, for contradiction, that $F(t^3) = T_1(t) T_2(t)$ where $T_1(t)$ is a monic irreducible polynomial of degree $2$, and $T_2(t)$ is a polynomial of degree $4$ in $\R[t]$.
Note that $T_2$ cannot have a degree greater than $4$, because the leading coefficient of $T_1$ is a unit, and thus the leading term of $F(t^3)$ would then have degree strictly greater than $6$, contradicting the degree of $F(t^3)$.
Without loss of generality, we may assume that both $T_1$ and $T_2$ are monic; indeed, since the leading coefficient of $F$ is $1$, if $T_1$ and $T_2$ are not monic, one can apply the inverse of the leading coefficients to make them monic.
This gives $f(t^3)=\overline{T_1}(t)\overline{T_2}(t)=(t^2+1)^3$. As the degrees of $\overline{T_1}$ and $\overline{T}_2$ must be $2$ and $4$ respectively, and $\mathbb{F}_3[t]$ is a unique factorization domain, $\overline{T}_1(t)=t^2+1$ and $\overline{T}_2(t)=(t^2+1)^2$. 
Therefore $T_1(t)=t^2+1+m_1(t)$ and $T_2(t)=(t^2+1)^2+m_2(t)$ for some $m_1(t),m_2(t)\in\mathfrak{m}[t]$, with $\deg(m_i)<\deg(T_i)$ for $i=1,2$. 
Let $m_1(t)=a_1t+a_2$ and $m_2(t)=b_1t^3+b_2t^2+b_3t+b_4$ where $a_i,b_j\in\mathfrak{m}$ for $i=1,2$, $j=1,2,3,4$.
As $T_1(t)T_2(t)=t^6-3t^3-5\in \R[t]$, comparing the coefficients we obtain 
\begin{align}
\text{Comparing the coefficient of }t^5 &: b_1 + a_1 = 0 \label{eq:1} \\
\text{Comparing the coefficient of }t^4 &: 3 + b_2 + a_1 b_1 + a_2 = 0 \label{eq:2} \\
\text{Comparing the coefficient of }t^3 &: b_3 + 2a_1 + a_1 b_2 + b_1 + a_2 b_1 = -3 \label{eq:3} \\
\text{Comparing the coefficient of }t^2 &: 1 + b_4 + a_1 b_3 + 2 + b_2 + 2a_2 + a_2 b_2 = 0 \label{eq:4} \\
\text{Comparing the coefficient of }t &: a_1 + a_1 b_4 + b_3 + a_2 b_3 = 0 \label{eq:5} \\
\text{Constant term coefficient} &: (1 + a_2)(1+b_4) = -5 \label{eq:6}
\end{align}
From \cref{eq:6}, the possible $(a_2,b_4)$ pairs are as follows
\begin{table}[h]
\centering

\begin{tabular}{ccc}
\toprule 
\multicolumn{3}{c}{$(a_2, b_4)$ pairs} \\
\midrule 
$(0, 3)$ & $(3, 0)$ & $(6, 6)$ \\
\bottomrule 
\end{tabular}

\end{table}\\
When $a_2=0$, \cref{eq:1,eq:2} imply that $b_2=a_1^2-3$ and hence the possible $(a_1,b_1,b_2)$ tuple are as follows:
\begin{table}[h]
\centering

\begin{tabular}{ccc}
\toprule 
\multicolumn{3}{c}{$(a_1, b_1,b_2)$ tuple for $(a_2,b_4)=(0,3)$} \\
\midrule 
$(0, 0,6)$ & $(3, 6,6)$ & $(6,3, 6)$ \\
\bottomrule 
\end{tabular}

\end{table}\\
Similarly, one achieves 
\begin{table}[h]
\centering

\begin{tabular}{ccc}
\toprule 
\multicolumn{3}{c}{$(a_1, b_1,b_2)$ tuple for $(a_2,b_4)=(3,0)$} \\
\midrule 
$(0, 0,3)$ & $(3, 6,3)$ & $(6,3, 3)$ \\
\bottomrule 
\end{tabular},
\begin{tabular}{ccc}
\toprule 
\multicolumn{3}{c}{$(a_1, b_1,b_2)$ tuple for $(a_2,b_4)=(6,6)$} \\
\midrule 
$(0, 0,0)$ & $(3, 6,0)$ & $(6,3, 0)$ \\
\bottomrule 
\end{tabular}

\end{table}\\

From \cref{eq:3} we obtain the corresponding values of $b_3$ and we have achieved the following data set:
\begin{table}[H]
\begin{tabular}{|l|l|l|l|l|l|}
\hline
$a_1$ & $a_2$ & $b_1$ & $b_2$ & $b_3$ & $b_4$ \\ \hline
0    &      & 0    & 6    & 6    &      \\ \cline{1-1} \cline{3-5}
3    & 0    & 6    & 6    & 3    & 3    \\ \cline{1-1} \cline{3-5}
6    &      & 3    & 6    & 0    &      \\ \hline
0    &      & 0    & 3    & 6    &      \\ \cline{1-1} \cline{3-5}
3    & 3    & 6    & 3    & 3    & 0    \\ \cline{1-1} \cline{3-5}
6    &      & 3    & 3    & 0    &      \\ \hline
0    &      & 0    & 0    & 6    &      \\ \cline{1-1} \cline{3-5}
3    & 6    & 6    & 0    & 3    & 6    \\ \cline{1-1} \cline{3-5}
6    &      & 3    & 0    & 0    &      \\ \hline
\end{tabular}
\end{table}
If we consider the tuple $(a_1, a_2, b_1, b_2, b_3, b_4)$, then in the cases where $b_3 = 6$, the values fail to satisfy \cref{eq:4}, and in all other cases where $b_3 = 3$ or $0$, the values fail to satisfy \cref{eq:5}. Hence, \ul{no such $T_1(t)$ exists as a degree $2$ irreducible factor of $F(t^3)$}.

Next, we present an example to illustrate that the assumption $\gcd(L, p) = 1$ is essential for the validity of \Cref{thm:cyc}.
\begin{example}
    We work with the ring $\R=\mathbb{Z}/9\mathbb{Z}$, $k\cong\mathbb{F}_3$, $L=3$ and $n=2$, so that $\gcd(L,p)=3\neq 1$. 
Consider the matrix $A= \begin{pmatrix}
 1 & 1 \\
0 & 1
\end{pmatrix} \in \GL_2(\mathbb{Z}/9\mathbb{Z})$.
\end{example}
Since the matrix $\overline{A}=\begin{pmatrix}
    1 & 1\\ 0 & 1
\end{pmatrix}\in\GL_2(\mathbb F_3)$ is not a member of $\im(\overline{\Phi}_3)$, it follows that $A\not\in\im(\Phi_3)$.
But $\Char_{\R,A}(t)=(t-1)^2$. Setting $F(t)=t-1$, $F(t^3)$ has an irreducible factor of degree $1$. 
Thus, an analogous statement to \Cref{thm:cyc} need not hold when we omit the condition $\gcd(L,p)=1$.

\section{Concluding remarks}\label{res-conc-rem}
\subsection{On $L$-th powers in $\GL_n(\R)$ and $\GL_n(k)$}
As mentioned in the introduction, our motivation for this work arises from questions concerning word problems in matrix groups over a local principal ideal ring.
As demonstrated in \Cref{exm:down-not-up}, it is not always the case that a matrix $A \in \GL_n(\R)$ is an $L$-th power if and only if its reduction $\overline{A} \in \GL_n(k)$ is an $L$-th power.
However, this phenomenon does hold for all $A \in \GL_n(\R)$ such that $\overline{A} \in \GL_n(k)$ is either regular semisimple or cyclic.
It would therefore be desirable to classify all elements of $\GL_n(\R)$ that are an $L$-th power precisely when their $\mathfrak m$-reduction in $\GL_n(k)$ is an $L$-th power.
\subsection{A question of Ofir Gorodetsky
} Gorodetsky asked the following in 2018; see \cite{GorodetskyOverflow}:
\emph{``We have a structure theorem for finitely generated modules over $R$, whenever $R$ is a PID. In the case of $R=\mathbb{Z}/p\mathbb{Z}[x]$ ($p$ a prime), the structure theorem can be used to obtain the rational canonical form for matrices over the finite field $\mathbb{Z}/p\mathbb{Z}$. I am interested in some kind of canonical form for matrices over $\mathbb{Z}/p^k\mathbb{Z}$. Is there such a canonical form in the literature?
This question naturally leads to a more concrete one: What is known about the structure of f.g. modules over $\mathbb{Z}/p^k\mathbb{Z}[x]$?
If I have a matrix $A \in \mathrm{Mat}_n(\mathbb{Z}/p^k\mathbb{Z})$, the relevant $\mathbb{Z}/p^k\mathbb{Z}[x]$-module is the ``vector space'' $(\mathbb{Z}/p^k\mathbb{Z})^n$, on which $x$ acts as multiplication by $A$.”}

In the course of proving \Cref{thm:reg-sem-poly}, we have in fact obtained a canonical form for matrices whose reduction modulo $\mathfrak m$ is a regular semisimple element; this is implicit in the proof. Moreover, as presented in \Cref{sec:cyc}, we have established a canonical form for compatible cyclic matrices over $\R$. 
However, the same proof shows that these types of canonical forms can be achieved over a finite principal ideal ring $\mathscr O_\ell$ of length $\ell$; more precisely
\begin{enumerate}
    \item Let $A\in\M_n(\mathscr O_\ell)$ be regular semisimple (i.e. its mod-$\mathfrak m$ reduction in $\M_n(k)$ is regular semisimple). 
    Then $A$ is conjugate to a matrix of the form $\mathrm{diag}(C_{F_{1}},C_{F_2},\ldots,C_{F_r})\in\M_n(\mathscr O_\ell)$, where $F_i$s are unique monic fundamental irreducible factor of $\mathrm{Min}_{\mathscr O_\ell,A}(t)$, and
    \item Let $A\in\M_n(\mathscr O_\ell)$ be a compatible cyclic matrix (see \Cref{sec:cyc}). 
    Then $A$ is conjugate to a matrix of the form $\mathrm{diag}(J_{\mathscr O_\ell,F_{1}}(r_1),J_{\mathscr O_\ell,F_{2}}(r_2),\ldots,J_{\mathscr O_\ell,F_{s}}(r_s))$ where $\Char_{\mathscr O_\ell,A}(t)=\mathrm{Min}_{\mathscr O_\ell,A}(t)=\prod\limits_{i=1}^sF_i(t)^{r_i}\in\mathscr O_\ell[t]$.
\end{enumerate}
We hope this line of investigation can be pursued further to develop canonical forms for matrices of other types.

\section*{Declarations}

% \noindent\textbf{Funding.}  

\noindent\textbf{Conflict of interest.} The authors declare that they have no conflict of interest.  

\noindent\textbf{Availability of data and materials.} Not applicable.  

\noindent\textbf{Code availability.} No codes were used during the preparation of the results.

\noindent\textbf{Ethics approval.} Not applicable.  

\noindent\textbf{Consent to participate.} Not applicable.  

\noindent\textbf{Consent for publication.} Not applicable. 

\noindent\textbf{Contribution statement.} All authors contributed equally to every aspect of this article. The order of authors is alphabetical by surname.
\printbibliography
\vspace{2em}
\end{document}